\documentclass[12pt,a4paper,oneside,onecolumn]{article}
\usepackage[T1]{fontenc}
\usepackage[utf8]{inputenc}
\usepackage[english]{babel}
\usepackage{amsmath, amssymb}
\usepackage{amsthm}
\usepackage{amsfonts}
\usepackage{amsrefs}
\usepackage{verbatim}
\usepackage{mathrsfs}
\usepackage{bbm}

\let\oldbf\bfseries
\renewcommand*{\bfseries}{\oldbf\mathversion{bold}\relax}
\let\oldnorm\normalfont
\renewcommand*{\normalfont}{\oldnorm\mathversion{normal}}
\let\oldrm\rmfamily
\renewcommand*{\rmfamily}{\oldrm\mathversion{normal}}

\usepackage{a4}
\usepackage{float}
\usepackage[nodayofweek,level]{datetime}
\usepackage{titling}
\usepackage{geometry}
\usepackage{graphicx}
\usepackage{pgf,tikz}
\usetikzlibrary{arrows}
\usetikzlibrary{patterns}
\usepackage{enumitem}
\usepackage{caption}
\usepackage{subcaption}
\usepackage{hyperref}
\hypersetup{
    colorlinks,
    linkcolor=red,
    citecolor=green,
    urlcolor=blue
}
\usepackage{breakurl}
\newtheorem{thm}{Theorem}
\newtheorem{cor}[thm]{Corollary}
\newtheorem{lem}[thm]{Lemma}
\newtheorem{prop}[thm]{Proposition}
\newtheorem{conj}[thm]{Conjecture}
\newtheorem{ques}[thm]{Question}
\theoremstyle{definition}
\newtheorem{defn}[thm]{Definition}
\newtheorem{rem}[thm]{Remark}

\newtheorem{obs}[thm]{Observation}

\numberwithin{thm}{section}
\def\cC{\mathcal{C}}

\def\cE{\mathcal{E}}

\def\cL{\mathcal{L}}

\def\cU{\mathcal{U}}

\def\Ex{\mathbb{E}}

\def\H{\mathbb{H}}

\def\N{\mathbb{N}}
\def\Pr{\mathbb{P}}

\def\R{\mathbb{R}}

\def\bbT{\mathbb{T}}

\def\Z{\mathbb{Z}}

\def\1{\mathbbm{1}}
\def\<{\langle}
\def\>{\rangle}

\renewcommand{\leq}{\leqslant}

\renewcommand{\le}{\leqslant}
\renewcommand{\ge}{\geqslant}
\renewcommand{\to}{\rightarrow}

\newcommand{\pc}{{p_{\mathrm{c}}}}
\newcommand{\qc}{{q_{\mathrm{c}}}}

\newcommand{\qct}{{\tilde q_{\mathrm{c}}}}

\newcommand{\qcu}{{\qc(\cU)}}
\newcommand{\pcop}{{p_{\mathrm{c}}^{\mathrm{OP}}}}
\newcommand{\pcsp}{{p_{\mathrm{c}}^{\mathrm{SP}}}}
\definecolor{qqffqq}{rgb}{0,1,0}
\definecolor{ffqqqq}{rgb}{1,0,0}
\definecolor{cqcqcq}{rgb}{0.75,0.75,0.75}

\title{$\cU$-bootstrap percolation: critical probability, exponential decay and applications}
\author{Ivailo Hartarsky\thanks{Département de Mathématiques et Applications, École Normale Supérieure, PSL University, 45 rue d’Ulm, 75005 PARIS -- \texttt{ivailo.hartarsky@dauphine.psl.eu}}}
\date{\today}
\begin{document}
\maketitle
\begin{abstract}
Bootstrap percolation is a wide class of monotone cellular automata with random initial state. In this work we develop tools for studying in full generality one of the three `universality' classes of bootstrap percolation models in two dimensions, termed subcritical. We introduce the new notion of `critical densities' serving the role of `difficulties' for critical models~\cite{Bollobas14}, but adapted to subcritical ones. We characterise the critical probability in terms of these quantities and successfully apply this link to prove new and old results for concrete models such as DTBP and Spiral as well as a general non-trivial upper bound. Our approach establishes and exploits a tight connection between subcritical bootstrap percolation and a suitable generalisation of classical oriented percolation, which will undoubtedly be the source of more results and could provide an entry point for general percolationists to bootstrap percolation.

Furthermore, we prove that above a certain critical probability there is exponential decay of the probability of a one-arm event, while below it the event has positive probability and the expected infection time is infinite. We also identify this as the transition of the spectral gap and mean infection time of the corresponding kinetically constrained model. Finally, we essentially characterise the noise sensitivity properties at fixed density for the two natural one-arm events.

In doing so we answer fully or partially most of the open questions asked by Balister, Bollobás, Przykucki and Smith~\cite{Balister16}---namely we are concerned with their Questions 11, 12, 13, 14 and 17.
\end{abstract}
\newpage
\tableofcontents
\newpage
\section{Introduction}
\subsection{Background}
The bootstrap pecolation (BP) process is a deterministic monotone cellular automaton first introduced in 1979 by Chalupa, Leath and Reich~\cite{Chalupa79} (see \cite{Morris17} for a review). Given a set $A\subset \Z^d$ or $(\Z/n\Z)^d$ of initially infected vertices (sites), we declare more vertices to be infected on each (discrete) time step according to a local rule. For a given initial state $A$, we say that BP occurs if every vertex eventually becomes infected. In the first examples considered a site becomes infected if at least $r$ of its nearest neighbours are already infected. These models are motivated by several different facets of statistical physics (see e.g.\ \cites{Adler91,Adler03}). For instance, they can represent nucleation or excitation of a metastable material. Moreover, they are tightly related to the zero-temperature dynamics of the Ising model \cites{Fontes02, Morris11}, as well as kinetically constrained models for the liquid-glass transition \cites{Cancrini08,Hartarsky20FA}. In these applications and the vast majority of BP literature the initial set $A$ is chosen randomly according to a product Bernoulli measure with density of infections $q$, which we denote $\Pr_q$. A quantity of major interest for this model is its critical probability defined by
\[\qc=\inf\left\{q,\Pr_q([A]=\Z^d)\ge 1/2\right\}\]
on $\Z^2$ and similarly for other graphs. We denote the parameter $q$ instead of the standard $p$, as it will become clear that a more natural description of the model is in terms of a certain infinite `cluster' of healthy sites, whose density is $1-q$. We will use the term \emph{percolation} for a random subset of $\Z^2$ with law $\Pr_p$ for any $p$ without necessarily referring to a BP process. 

The first results on BP due to van Enter~\cite{VanEnter87} and Schonmann~\cite{Schonmann92} proved the triviality of the phase transition for all values of the parameters $r$ and $d$. However, Aizenmann and Lebowitz~\cite{Aizenman88} showed that when the dynamics is considered on a finite box $\{1,\dots,n\}^d$ instead of $\Z^d$, the critical probability scales like $\Theta\left((\log n)^{1-d}\right)$ for the nearest neighbour model with $d\ge r=2$. As it was noticed by Balogh and Bollobás~\cite{Balogh03} the phase transition is sharp owing to the general result of Friedgut and Kalai\cite{Friedgut96}. The position of the sharp threshold for $d=r=2$ was determined in a breakthrough of Holroyd~\cite{Holroyd03}. His results were then improved further and now the scaling of the second term of the critical probability is exactly known~\cites{Gravner08,Hartarsky19} in this setting. For $d\ge r>2$ the correct scaling was determined by Cerf and Cirillo~\cite{Cerf99} and Cerf and Manzo~\cite{Cerf02}. The corresponding sharp threshold was established by Balogh, Bollobás and Morris and the same authors together with Duminil-Copin \cites{Balogh09a,Balogh12}.

However, the methods of those works remained highly model-dependent, while many more models had been studied in the literature and some exhibited very different behaviour~\cites{Schonmann92,Mountford95,VanEnter07,Gravner96}. A relatively general classification was first attempted by Gravner and Griffeath~\cites{Gravner96,Gravner99}. It was much later substantially generalised, rectified and universality results were rigorously proved by Bollobás, Smith and Uzzell~\cite{Bollobas15} and Balister, Bollobás, Przykucki and Smith~\cite{Balister16}. It is this vast class of models that we introduce now. Although much of our work easily carries over to higher dimensions, we restrict ourselves to models on $\Z^2$, as the universality picture is currently only established in this setting.

\subsection{Models}
\label{subsec:models}
\paragraph{Bootstrap percolation} A BP model is parametrised by an \emph{update family}---a finite family $\cU$ of finite non-empty subsets of $\Z^2\setminus \{0\}$ called \emph{rules}. The initial set of infections $A=A_0$ is taken at random according to the product Bernoulli measure with parameter $q$, $\Pr_q$, and we define the evolution of the dynamics by
\[A_{t+1}=A_t\cup\{x\in\Z^2,\;\exists U\in\cU,\,x+U\subset A_t\},\]
so that a site becomes infected if any of the rules is entirely infected already. We denote by $[A]=\bigcup_{t\ge 0}A_t$ the \emph{closure} of $A$ and extend this notation to $A\subset\R^2$ by setting $[A]:=[A\cap\Z^2]$.

The result of \cites{Bollobas15,Balister16} is a partition of these models into three classes. The classification is based on the notion of stable directions---a direction $u\in S^1=\{x\in\R^2,\|x\|_2=1\}$ is \emph{unstable} if there exists $U\in\cU$ entirely contained in the half-plane $\H_u=\{x\in\R^2,\<x,u\><0\}$ and \emph{stable} otherwise, where $\<\cdot,\cdot\>$ denotes the canonical scalar product of $\R^2$. In terms of the BP process $u\in S^1$ is stable iff $[\H_u]=\H_u\cap\Z^2$ and unstable iff $[\H_u]=\Z^2$ (see \cite{Bollobas15}*{Lemma 3.1}). With this terminology BP models are classified as follows.
\begin{itemize}
\item \emph{Supercritical} if there exists an open semi-circle of unstable directions. In this case $\qc((\Z/n\Z)^2)=n^{-\Theta(1)}$ \cite{Bollobas15}.
\item \emph{Critical} if there exists a semi-circle with a finite number of stable directions, but it is not supercritical. In this case $\qc((\Z/n\Z)^2)=(\log n)^{-\Theta(1)}$ \cite{Bollobas15}.
\item \emph{Subcritical} otherwise, i.e.\ if every semi-circle contains infinitely many stable directions. In this case on $\Z^2$ we have $\qc>0$ \cite{Balister16}.
\end{itemize}
It is not hard to check (see \cite{Bollobas15}*{Theorem 1.10}) that, since the update family and rules are finite, the set of stable directions forms a finite union of closed intervals of $S^1$, with rational endpoints, a direction $u\in S^1$ being rational if there exists $x\in\Z^2$ such that $x=\lambda u$ for some $\lambda\in\R\setminus\{0\}$. In particular, the above classification is indeed equivalent to \cite{Bollobas15}*{Definition 1.3}.

Before discussing further results, let us give a few examples to digest these definitions. Turning back to the $r$-neighbour models, one can embed them in this setting by considering $\cU$ consisting of all $r$-element subsets of the set of the $4$ nearest neighbours of $0$. The case $r=1$ is supercritical, as it has no stable directions; $r=2$ is critical, as only the four lattice directions are stable; $r\in\{3,4\}$ are subcritical, as they have no unstable directions.

\paragraph{Trivial subcritical models} Models with no unstable directions, which  we call \emph{trivial subcritical models}, are not particularly relevant for us. It can be shown that they are exactly the models such that $\qc=1$ or, equivalently, such that there exist finite sets of healthy sites, which cannot be infected by the rest of $\Z^2$ \cite{Balister16}. Therefore, it is useful to introduce more interesting subcritical models, which will be investigated further in this work.

\paragraph{Oriented site percolation} \emph{OP} can be viewed as the BP model with $\cU=\{\{(1,1),(-1,1)\}\}$. It is easy to check (and was noticed already by Schonmann \cite{Schonmann92}) that $x\not\in[A]$ if and only if there exists an infinite oriented path (with North-East and North-West steps) of initially healthy sites. In particular $\qc$ for this model is equal to $1-\pcop$, where $\pcop$ is the usual critical probability of OP parametrised in terms of the density of healthy sites (this is one of the reasons for denoting our parameter $q$). Up to applying an invertible linear transformation to $\Z^2$, any family with one rule consisting of two non-collinear sites is equivalent to OP, so we will abusively also call them OP. Furthermore, one may consider \emph{bidirectional OP} with $\cU'=\{\{(1,1),(-1,1)\},\{(-1,-1),(1,-1)\}\}$, for which the surviving healthy sites are those initially belonging to a bi-infinite oriented path, so that the critical probability is again $1-\pcop$. OP is a very classical and well-understood model, for background on which we direct the reader to \cite{Durrett84} in addition to Section \ref{sec:OP}.

\paragraph{One-rule families and generalised oriented percolation} More generally, it is natural to consider all update families with only one rule, $\cU=\{U\}$. There are three types of them. If the origin is contained in the convex envelope of $U$, then $\cU$ is trivial subcritical. If the rule is contained in a half-line starting at the origin, the model is supercritical. All remaining one-rule families give rise to non-trivial subcritical models and their rule is contained in an open half plane. We refer to these models as \emph{generalised oriented site percolation} (GOP). Specific cases of these correspond to well-known probabilistic cellular automata (see \cite{Hartarsky20GOP}).

\paragraph{Spiral model} The \emph{Spiral model} of Toninelli, Biroli and Fisher \cite{Toninelli07a} is defined by $\cU=\{U_1,U_2,U_3,U_4\}$, where
\begin{equation}
\begin{aligned}
U_1&=\{(1,-1),(1,0),(1,1),(0,1)\} &U_2&=\{(1,-1),(1,0),(-1,-1),(0,-1)\}\\
U_3&=-U_1,  &U_4&=-U_2.
\end{aligned}
\label{eq:def:Spiral}
\end{equation}
This model was introduced to witness the somewhat surprising fact that subcritical BP can have a discontinuous phase transition in the sense that $\theta(\qc)=\Pr_{\qc}(0\not\in[A])>0$. This was established rigorously by Toninelli and Biroli \cite{Toninelli08} based on a close relationship with OP, which we will discuss further in Section \ref{sec:app}.

\paragraph{Directed triangular bootstrap percolation} \emph{DTBP} was introduced by Balister, Bollob\'as, Przykucki and Smith \cite{Balister16} as an example of a simple, but somewhat generic, subcritical model. Its main feature is its lack of symmetry and it should be viewed as a benchmarking example. It can be defined as $2$-neighbour BP on a directed triangular lattice, but can also be embedded in $\Z^2$ by
\begin{equation}
\cU=\{\{(1,0),(0,1)\},\{(1,0),(-1,-1)\},\{(0,1),(-1,-1)\}\}.
\label{eq:def:DTBP}
\end{equation}
As for most subcritical models essentially nothing is known about it. As a quantitative illustration of their result, the authors of \cite{Balister16} established that for DTBP
\[10^{-101}\le \qc\le 1-\pcop\le 0.3118,\]
invoking Gray, Weirman and Smythe \cite{Gray80} for the last inequality.

\paragraph{Kinetically constrained models} \emph{KCM} are stochastic generalisations of BP, although they were introduced independently to model the liquid-glass transition \cite{Fredrickson84}. A KCM is defined by an update family $\cU$ and a parameter $q$ as in BP. It is a Markov process on $\{0,1\}^{\Z^2}$ reversible w.r.t.\ $\Pr_q$ with the following graphical representation (see e.g.\ \cite{Liggett05} for background on interacting particle systems). We consider independent standard Poisson processes on each site and at each point of those processes the state of the corresponding site is resampled from its equilibrium Bernoulli measure with parameter $q$ if it would become infected on the next step in the BP process with the same update family and remains unchanged otherwise. Cancrini, Martinelli, Roberto and Toninelli~\cite{Cancrini08} proved that the critical probability of a KCM (above which $0$ is a simple eigenvalue of the Markov generator) is equal to $\qc$ for the corresponding BP. Furthermore they proved, using a general halving technique, that the spectral gap of the Markov generators of specific KCM considered in the physics literature is strictly positive (see e.g.\ \cite{Hartarsky19I} for a variational definition).

\paragraph{Supercritical models}
Certain classes of supercritical BP were studied extensively in the 90s by Gravner and Griffeath (see \cite{Gravner98}), mainly from the point of view of limit shapes and minimal percolating sets, which are usually finite for such models. Their KCM counterparts offer additional complexity, but were successfully studied in general by Mar\^ech\'e, Martinelli, Morris and Toninelli in different combinations \cites{Martinelli19a,Mareche20combi,Mareche20Duarte}.

\paragraph{Critical models}
For general critical BP the most notable results are due to Bollobás, Duminil-Copin, Morris and Smith~\cite{Bollobas14}. They introduced a notion of `difficulty' of a (topologically) isolated stable direction, counting the number of additional infections needed for an infected half-plane to grow and used it to determine the exact scaling (up to a constant factor) of $\qc((\Z/n\Z)^2)$ for all critical models. Although it will not be used directly, this notion is an important inspiration for our work. Sharper results generalising the one of Holroyd~\cite{Holroyd03} were also proved in a more restrictive but still somewhat general framework by Duminil-Copin and Holroyd~\cite{Duminil-Copin12}.

Very recently, an equaly complete understanding of critical KCM was achieved in a series of works by Marêché, Martinelli, Morris, Toninelli and the author in various combinations~\cites{Martinelli19a,Hartarsky19I, Hartarsky19II,Hartarsky20I, Hartarsky20II,Hartarsky20FA}, establishing a more complex behaviour.

\paragraph{Subcritical models} The focus of our work are the least understood of the three classes of update families---subcritical ones. For them the only result in full generality to date is the one of Balister, Bollob\'as, Przykucki and Smith~\cite{Balister16} stating that $\qc>0$. The technique behind it is a fairly involved multi-scale renormalisation, which has little hope of providing more results than what Peierls arguments give for standard percolation models (see Grimmett's classical monograph on percolation theory \cite{Grimmett99}*{Chapter 1}). We should note that computing the critical probability $\qc$ of subcritical models does not seem plausible even for the simplest one---OP.\footnote{To quote Grimmett \cite{Grimmett99} in the setting of standard percolation, but also valid e.g.\ for OP, ``It is highly unlikely that there exists a useful representation of $\pc$ [...] although such values may be computed with increasing degrees of accuracy with the aid of larger and faster computers.''}

\paragraph{Ordinary site percolation} Finally, \emph{SP} is one of the most classical percolation models (see \cite{Grimmett99}), which will also be useful for us, although it is not a particular case of BP. Similarly to OP, it consists in declaring each site of $\Z^2$ open independently with probability $p$ and looking for infinite paths of open sites with respect to the usual nearest neighbour graph structure of $\Z^2$ instead of the oriented one for OP. We denote $\pcsp$ the critical probability of appearance of such infinite paths.

\section{Definitions and notation}
\label{sec:defs}
In this section we gather most of the notation used throughout the article. We invite the reader familiar with percolation to skip ahead to Section \ref{sec:results} and go back to this section as needed. As some of the notions will be used relatively locally, let us indicate that the central notion of the present work is the one in Definition \ref{def:diff}.

\paragraph{Critical probability} Recall that $0<q<1$ is the density of infected sites and $\Pr_q$ is the associated Bernoulli product law of the random set $A\subset\Z^2$ and that $[\cdot]$ denotes the closure with respect to the BP process defined by a non-trivial update family $\cU$, that we keep implicit when there is no risk of confusion. Also, $B_x=[-x,x]^2\cap\Z^2$ for all $x\in[0,\infty)$. Define
\begin{align*}
\theta_n(q)&=\Pr_q(0\not\in[A\cap B_n]),\\
\theta(q)&=\lim_{n} \theta_n(q)=\Pr_{q}(0\not\in[A]).
\end{align*}
The critical probability is given by
\[\qc=\inf\left\{q\in[0,1],\Pr_q([A]=\Z^2)=1\right\}=\sup\left\{q,\theta(q)>0\right\},\]
the first equality following from ergodicity and the second one resulting from invariance by translation as for SP (see e.g.\ \cite{Grimmett99}). We also introduce another critical probability
\begin{equation}
\label{eq:def:pct}
\qct=\inf\left\{q\in[0,1],\,\sum_n n\theta_n(q)<\infty\right\},
\end{equation}
which is actually the only relevant one for our proofs, only noting that $\qct\ge\qc$. Several other equivalent definitions will be proved in Theorem~\ref{th:exp:decay}, so that $\qct$ is in particular the critical probability of exponential decay of $\theta_n(q)$. We emphasise that working with $\qct$ instead of $\qc$ will only lead to stronger results in applications.

\paragraph{Directions and half-planes} In order to define the central notion of this work, critical densities, we will need some conventions and notation concerning directions and half-planes, which will mostly follow previous authors. We identify the unit circle $S^1\subset\R^2$ with the torus $\R/2\pi\Z$ via \[(\cos\theta,\sin\theta)\longleftrightarrow \theta\mod 2\pi.\]
Despite the identification we shall preferentially use the letters $u,v$ for directions and $\theta$ for angles. 
For $n\in \N$ directions $u_1,\ldots u_n\in S^1$ we write $u_1<\ldots <u_n$ if one can find $\theta_1<\ldots<\theta_n<\theta_1+2\pi$ and $\theta$ in $\R$ such that for each $i$ we have $u_i\longleftrightarrow (\theta_i-\theta) \mod 2\pi$.

Recall that $\langle\cdot,\cdot\rangle$ and $S^1$ are the canonical scalar product on $\R^2$ and its unit sphere (circle). Furthermore, for $u\in S^1$ and $a\in\R$ set
\[\H_u^a =\{v\in\R^2, \<v,u\>< a\},\]
$\H_u=\H_u^0$ and $\H_u^{a+}=\{v\in\R^2,\<v,u\>\le a\}$. We denote by $V_{u,v}=\H_u\cap \H_v$ the cone defined by the directions $u,v\in S^1$. We also recall the standard notation $a\vee b=\max(a,b)$ and $a\wedge b=\min(a,b)$.

\paragraph{Critical densities} We are now ready to  introduce the new notion of `critical densities' adapted to subcritical BP (for critical and supercritical ones they will turn out to be identically $0$). Let us note that this is not an extension, but rather a complement, of the `difficulties' of~\cite{Bollobas14}, which are trivial for subcritical models.

Before we frighten the reader with the definition, let us say that the critical density in a direction $u$ is morally the critical probability of the model with infected boundary condition in $\H_u$. The definition we give differs from this one in two ways---it concerns the critical probability for certain decay of $\theta_n(q)$ and it is defined in a region whose shape approaches a half-plane. Nevertheless, this distinction will only be of major importance for Section~\ref{subsec:main}. That is because in applications we will always rely on simple OP-like models, in which we know that there is exponential decay above criticality and that the critical density is continuous in the shape of the region, so that the two notions coincide. Finally, we actually conjecture that they are \emph{always} equal. With this in mind, let us state the definition we shall use.
\begin{defn}
\label{def:diff}
For $u\in S^1$ and $\theta\in[-\pi,\pi]$ define
\[d_{u}^{\theta} = \inf\left\{q\in[0,1],\,\sum_n n\Pr_q\left(0\not\in [((A\cup V_{u,u+\theta}\right)\cap B_n)])<\infty\right\}.\]
Taking the (monotone) limit of this quantity, we set
\[d_{u}^{\pm}=\lim_{\theta\to 0\pm}d_u^\theta\]
and we call $d_{u}^-$ and $d_{u}^+$ the \emph{left and right critical densities} of $u$ respectively. The \emph{critical density} of $u$ is then given by $d_u=d_u^+\vee d_u^-$. We call $u\mapsto d_u$ the \emph{critical density function} of the model (of $\cU$).
\end{defn}

It is clear from the definition that this quantity is somewhat of the same complexity as $\qc$, so that it is not feasible to be able to compute the critical densities for all $u$ even for the simplest of subcritical models---OP.

The next observation directly follows from Definition \ref{def:diff}, but will be the base for our upper bounds on $\qc$.
\begin{obs}
\label{obs:crit:dens}
Let $\cU$ be an update family. Let $u\in S^1$ be a direction and $\cU'\subset\cU$ be a subfamily of rules. Then
\[d_u(\cU)\le d_u(\cU').\]
\end{obs}

\paragraph{One-arm events} Generally in percolation theory, a one-arm event is an event corresponding to `a point being connected to infinity' or its finite-size truncations. In BP there is one very natural infinite volume one-arm event---$\{0\not\in[A]\}$, which corresponds to the presence of an infinite cluster (set) of healthy sites ensuring the occurrence of the event. There are several natural ways to truncate this event. In particular, we have 
\[\{0\not\in[A]\}=\bigcap_n\{\tau_0\ge n\}=\bigcap_n\{0\not\in[A\cap B_n]\},\]
etc., where $\tau_0$ is the infection time of the origin. We interpret this event as $0\to\infty$ (0 `looks at' infinity) and its truncated version $\{0\not\in[A\cap B_n]\}$ as $0\to\partial B_n$ ($\partial$ stands for the boundary). In models involving some kind of directionality, like BP, one may need to distinguish between `point-to-infinity' and `infinity-to-point' and similarly for truncated versions. The second one, which we define next, turns out to be more tractable, albeit less natural.

For $n\in\N$ and $x\in B_n$ we denote the infection time of $x$ in $B_n$ with healthy boundary condition by
\[\tau_x^{B_n}=\inf\left\{t,\,x\in(A\cap B_n)^{B_n}_t\right\},\]
where the dynamics only affects the configuration in $B_n$.
\begin{defn}
\label{def:En}
Fix a large constant $C>0$ depending on $\cU$. Denote by $E_n\subset \{0,1\}^{B_n}$ the event that there exists an integer $N$ and a sequence $(x_i)_0^N$ of sites in $B_n$ such that
\begin{itemize}
\item $x_N$ is at distance at most $C$ from the boundary $\partial B_n$ of $B_n$.
\item $x_0=0$
\item $x_{i-1}\in x_i+X$ for all $1\le i\le N$, where $X=\bigcup_{U\in\cU} U$
\item $\tau^{B_n}_{x_i}\ge i$.
\end{itemize}
Also set $\tilde\theta_n(q)=\Pr_q(E_n)$ and $\tilde\theta(q)=\lim_n\tilde\theta_n(q)$.
\end{defn}
Note that the healthy boundary condition does not influence this event too much. Indeed, it is clear that some $x_i$ is close to $\partial B_{n/2}$, so the occurrence of $E_n$ implies the existence of a site `in the bulk' (far from the boundary) with large infection time. We will use this observation to obtain information on the distribution of the infection time $\tau_0$ below $\qct$.

The events $E_n$, which we interpret as $\partial B_n\to 0$, have the notable advantage of being `reflexive' in the sense that, when exploring a configuration to check if $E_n$ holds, looking back at the explored region from its boundary, one sees the event itself occurring in a smaller domain, which is crucial for the argument of Duminil-Copin, Raoufi and Tassion \cite{Duminil-Copin19} that we will use. Also very importantly, this event is defined in terms of a path rather than a `cluster', although it does require the existence of `clusters' of healthy sites. Of course, the main disappointment is that although very closely related to (and only differing by at most polynomial factors from) the natural events $\{0\not\in[A\cap B_n]\}$ or $\{\tau_0\ge n\}$, it does not allow us to prove that $\qct=\qc$, but only provides additional constraints on the phase $[\qc,\qct)$. The reason is that we may have $\bigcap_n E_n\neq \{0\not\in[A]\}$, meaning that in BP the `$0\to\infty$' and `$\infty\to0$' events are different.

\paragraph{Randomised algorithms and revealment}
We will need the natural notion of \emph{algorithm determining} a random variable $Y$ on $\Omega_0=\{0,1\}^{B_n}$ endowed e.g.\ with the measure $\Pr_q$. Roughly speaking, this is an algorithm which reveals the state of one bit (the value of $\omega_0\in\Omega_0$ on one site $x\in B_n$) at a time possibly depending on knowledge of the configuration already explored. It keeps exploring bits one at a time until the value of $Y$ is witnessed by the explored sites (determined regardless of the state of the remaining unexplored sites).

More formally, an algorithm is a rooted strict binary tree $T$ directed away from the root. Its internal nodes are labelled by sites of $B_n$ indicating the state of which site is being revealed. For each such internal node labelled by $x$, the two out-edges are labelled by the two possible values of the corresponding bit, so that given $\omega_0\in\Omega_0$, the algorithm with input $\omega_0$ continues along the edge labelled by $\omega_0(x)$. The leaves of the tree are labelled by the possible values of $Y$ (with repetition) indicating which value of $Y$ is witnessed (guaranteed) by the states indicated by the edges from the root to the leaf. More precisely, let $P_l$ denote the path from the root to a leaf $l$ labelled by a possible value $y$ of $Y$. Then the vertices of $P_l$ all have distinct labels (each site is revealed at most once) and for any $\omega_0\in\Omega_0$ such that for all internal nodes $v\in P_l$ we have $\omega_0(x_v)=\epsilon_v$ implies that $Y(\omega_0)=y$, where $x_v$ is the label of $v$ and $\epsilon_v$ is the label of the out-edge of $v$ belonging to $P_l$. Clearly, given an algorithm and an input $\omega_0\in\Omega_0$, there exists a unique leaf $l_{\omega_0}$ such that for every internal node in $v\in P_{l_{\omega_0}}$ we have $\omega_0(x_v)=\epsilon_v$. This simply corresponds to what the algorithm actually does for the specific realisation of the random input---which sites it checks, in what order, what values it finds for their states and, finally, what value of the random variable $Y$ it determines based on those states. 

A \emph{randomised algorithm} is an algorithm-valued random variable. As we will apply these algorithms to inputs which are random themselves, we need to define them on a common probability space $(\Omega,\Pr)$, so that the random algorithm is independent from the random input. For a randomised algorithm define its \emph{maximal revealment}
\[\delta=\max_{x\in B_n}\Pr(\exists v\in P_{l_{\omega_0}}, x_v=x),\]
i.e.\ the maximal probability that any fixed site is explored by the algorithm.

\paragraph{Noise sensitivity}
We next define noise sensitivity, although our proofs will mostly use black-box theorems based on Fourier analysis instead of the definition.\begin{defn}
\label{def:noise}
Let $G_n\subset\{0,1\}^{B_n}$ be a sequence of events. For every $\omega_0\in\{0,1\}^{B_n}$ let $N_\varepsilon(\omega_0)$ be the configuration obtained when each bit of $\omega_0$ is resampled independently with probability $\varepsilon$ and unchanged otherwise. Resampled bits are taken to be independently infected with probability $q$ as originally.

We say that the sequence $G_n$ is \emph{noise sensitive}, if for every $\varepsilon>0$
\[\lim_{n\to\infty}\frac{Cov\left(\1_{\omega_0\in G_n},\1_{N_\varepsilon(\omega_0)\in G_n}\right)}{Var(\1_{G_n})}=0.\]
\end{defn}
Let us note that this definition following \cite{Bartha15} is stronger than the original one from~\cite{Benjamini99}, which is trivial for events with probabilities tending to $0$ and equivalent, if the probabilities are bounded away from $0$.

\section{Results}
\label{sec:results}
Our goal is to provide a toolbox for studying subcritical models in full generality. Although our results will apply also to supercritical and critical models, most of them are either empty or relatively easy for such families. Unless explicitly mentioned we do not consider trivial subcritical models.

\paragraph{Critical densities and upper bounds on $\qc$}
Let $\cC=\{[u,u+\pi],u\in S^1\}$ be the set of closed semi-circles of $S^1$. The most central result of our work is the following directional decomposition of the critical probability.
\begin{thm}
\label{th:general}
Let $\cU$ be any update family. Then
\begin{equation}
\label{eq:qct:decomposition}
\qct=\sup_{u\in S^1}d_u=\inf_{C\in\cC}\sup_{u\in C}d_u.
\end{equation}
If $\cU$ is not subcritical, then $\qct=0$.
\end{thm}

Combining Theorem~\ref{th:general} with Observation \ref{obs:crit:dens}, we obtain the following upper bound on $\qc$.
\begin{cor}
\label{cor:upper:bound}
Let $\cU$ be an update family. Then for any set of subfamilies $\cU_i\subset\cU$ we have
\[\qc(\cU)\le\qct(\cU)\le \inf_{C\in\cC}\sup_{u\in C}\min_{i}d_u(\cU_i).\]
\end{cor}

\paragraph{Critical densities of OP}
In order to make use of Corollary~\ref{cor:upper:bound} and obtain a concrete non-trivial upper bound in relative generality, we express the critical densities of OP in terms of a classical quantity called `edge speed'. This is done in Section \ref{sec:OP} by combining many standard facts about OP recalled there together with the definition of the `edge speed'.

\paragraph{Application to DTBP}
Though simple, the bound in Corollary \ref{cor:upper:bound} is very versatile and can lead to non-trivial results for the right choice of subfamilies we have information for. Of course, in some cases it will reduce to the trivial bound $\qc(\cU)\le\min_{U\in\cU} q_{\mathrm{c}}(\{U\})$ (since it is sometimes sharp already), which has not been brought up explicitly in the literature, but was mentioned for DTBP in \cite{Balister16}, taking only $\cU_1=\{U\}$ for some rule $U\in\cU$ (they are all isomorphic). There it was observed that $\qc\le 1-\pcop<0.312$, the second inequality being due to Gray, Weirman and Smythe \cite{Gray80}. 

As an exemplary application of our result, we improve this bound on DTBP, answering Question~17 of \cite{Balister16} (of course, the question may now be reiterated). We prove the following by combining Corollary~\ref{cor:upper:bound}, the expression of critical densities of OP and a variant of the argument from \cite{Gray80}.
\begin{thm}
\label{th:BBPS:family}
For DTBP
\[\qc\le\qct\le d_{\arctan(-1/3)}^{\mathrm OP}< 0.2452,\]
where $d^{\mathrm OP}$ is the critical density of OP.
\end{thm}

\paragraph{Application to Spiral}
Another application concerns the Spiral model. For that model Toninelli and Biroli \cite{Toninelli08} proved that $\qc=1-\pcop$, there is exponential decay for $q>\qc$ and its transition is discontinuous, as well as providing bounds on the exponentially diverging correlation length. It turns out that our method exactly recovers the first two assertions, giving a new proof of the following.
\begin{thm}[Theorem 3.3. of \cite{Toninelli08}]
\label{th:Spiral}
For the Spiral model $\qc=\qct=1-\pcop$.
\end{thm}
This is a consequence of Corollary~\ref{cor:upper:bound} together with an adaptation of a straightforward but fundamental lemma from \cite{Toninelli08}, which inputs a crucial feature of the model identified by Jeng and Schwarz~\cite{Jeng08}.

\paragraph{Exponential decay}
In the proof of Theorem~\ref{th:general} we actually prove that $\theta_n(q)$ decays exponentially fast in $n$ for $q>\qct$. We provide a second proof of this fact, which also gives additional information on the phase $q<\qct$.
\begin{thm}
\label{th:exp:decay}
Recalling Definition \ref{def:En}, for any update family the following holds.
\begin{itemize}
\item If $q>\qct$, then there exists $c(q)>0$ such that
\[\max\left(\theta_n(q),\tilde\theta_n(q)\right)\le \exp(-c(q).n).\]
\item There exists $c>0$ such that for $q<\qct$
\[\tilde\theta(q)\ge c.(\qct-q)>0.\]
\item If $q<\qct$, then there exists $c(q)>0$ such that
\[\Pr_q(\tau_0>n)\ge c(q)/n\]
and in particular $\Ex_q[\tau_0]=\infty$.
\end{itemize}
\end{thm}
Although we expect that $\qc=\qct$, this implies that if $\qc\neq\qct$, then the expected infection time is infinite at $\qc$ (Question~11 of~\cite{Balister16}).

The proof relies heavily on the new simple but powerful method of Duminil-Copin, Raoufi and Tassion~\cite{Duminil-Copin19} based on randomised algorithms. With some additional work on their only model-dependent Lemma~3.2, somewhat surprisingly the technique applies to BP, which is a rather unconventional setting for such arguments from SP.

Finally, we answer Question~12 of~\cite{Balister16} on exponential decay for $q<\qc$ in the negative and provide satisfactory information concerning Question~14 of the same paper on the relationship between BP and SP.

\paragraph{Noise sensitivity}
Exploiting the algorithm we devise in order to prove Theorem \ref{th:exp:decay}, we obtain the following relatively complete information about noise sensitivity.
\begin{thm}\label{th:noise}
Recalling Definition \ref{def:En}, for any update family and any $q\in(0,1)$ the following hold.
\begin{itemize}
\item $\tilde\theta(q)=0$ if and only if the events $E_n$ are noise sensitive and if and only if there is an algorithm with vanishing revealment determining their occurrence.
\item If $\theta(q)>0$, then the events $\{0\not\in[A\cap B_n]\}$ are not noise sensitive.
\item If $\theta(q)=\tilde\theta(q)=0$, then the events $\{0\not\in[A\cap B_n]\}$ are noise sensitive and there is an algorithm with vanishing revealment determining their occurrence.
\end{itemize}
\end{thm}
The proof relies on fundamental results of Benjamini, Kalai and Schramm~\cite{Benjamini99} and Schramm and Steif~\cite{Schramm10}.

In particular, this proves that Spiral is not noise sensitive at criticality, while OP is, so that the conditions on continuity of the transition are indeed relevant for noise sensitivity. Let us also mention that proving that the missing case---$\tilde\theta(q)>0=\theta(q)$---never occurs is only slightly stronger than proving Conjecture~\ref{conj:qc:qct} stating that $\qc=\qct$. If it indeed does not occur, then Theorem~\ref{th:noise} provides the final answer to Question~13 of \cite{Balister16} as far as one-arm events are concerned. Furthermore, Theorem~\ref{th:noise} suggests some limitations  for the intuition given by Bartha and Pete~\cite{Bartha15} (see Question~1.3 therein). Namely, Theorem~\ref{th:noise} indicates that noise sensitivity non-trivially depends on the continuity of the transition, while~\cite{Bartha15} suggests that it should only depend on whether the model is subcritical or not, though for a more restrained class of models. Therefore, if a variant of Question~1.3 of~\cite{Bartha15} is to hold in general, additional ramifications should be needed.

\paragraph{Spectral gap and mean infection time of KCM}
Another application of our exponential decay results concerns KCM. We extend to full generality the scope of the main result of Cancrini, Martinelli, Roberto and Toninelli~\cite{Cancrini08} using their method together with exponential decay.
\begin{thm}
\label{th:gap}
Consider any KCM. If $q<\qct$, then the spectral gap of its generator is $0$ and the mean infection time of the origin in the stationary process (with initial law $\Pr_q$) is infinite. If $q>\qct$, then the spectral gap is strictly positive and the mean infection time of the origin in the stationary process is finite.
\end{thm}
In other words, $\qct$ is the phase transition of the spectral gap of the associated KCM, so that it can be directly read off the associated BP as is the case of the non-ergodicity transition occurring at $\qc$~\cite{Cancrini08}.

We should note that the statement in the case of supercritical and critical models (for which $\qct=0$ by Theorem~\ref{th:general}) is also a trivial consequence of the quantitative result of~\cite{Martinelli19a}. We are particularly indebted to Cristina Toninelli for discussions around this theorem and its proof.

\section{Critical densities}
\label{sec:critdens}
In this section, after some short preparatory work of establishing basic properties of critical densities, we characterise $\qct$ in terms of them, which can be viewed as the most central result of the paper.
\subsection{Preliminaries}
We start with a few observations which follow trivially from Definition~\ref{def:diff}, but are essential nonetheless.
\begin{obs}
\label{obs:diff:monotone}
For all $u,\theta\in S^1$ one has
\[d_u^\theta\le \qct\]
and therefore the same holds for $d_u$ and $d_{u}^{\pm}$. Moreover, $\theta\mapsto d_u^{\theta}$ is non-decreasing for $\theta\in[0,\pi]$ and non-increasing for $\theta\in[-\pi,0]$ and $d_u^{\pm\pi}=\qct$.
\end{obs}

\begin{obs}
\label{obs:diff:sym}
For all $u,\theta\in S^1$ one has
\[d_u^\theta=d_{u+\theta}^{-\theta}.\]
\end{obs}

The following fundamental lemma is based on a classical topological trick.
\begin{lem}
\label{lem:topo}
Let $\varepsilon>0$ and $I\neq S^1$ be a closed interval of $S^1$, which we identify with an interval $[u,v]$ of $\R$. Then there exists $n\in\N$ and a finite sequence $u=u_0<u_1<\ldots <u_n=v$ of directions in $I$ such that
\begin{equation}
\label{eq:lem:topo}\forall i\in[1,n],\,0\le d_{u_{i-1}}^{u_i-u_{i-1}}-\left(d_{u_{i-1}}^+\vee d_{u_i}^-\right)<\varepsilon.
\end{equation}
\end{lem}
\begin{proof}
Recall that by Observation~\ref{obs:diff:sym} for $u',v'\in S^1$ with $0<v'-u'<\pi$ we have $d_{u'}^{v'-u'}=d_{v'}^{-(v'-u')}$. Then by Observation~\ref{obs:diff:monotone} one always has $d_{u'}^{v'-u'}\ge d^+_{u'}\vee d^-_{v'}$, so we need only establish the second inequality. 

Set 
\[I_0=\left\{v'\in[u,v],\,\exists n\exists(u_i)\in(S^1)^{n+1},\;u=u_0<\ldots <u_n=v',\,\text{satisfying \eqref{eq:lem:topo}}\right\},\]
and $v_0=\sup I_0$, which we shall prove to be $v$. To do this we prove that $I_0$ is open to the right:
\[\forall v'\in I_0\,\exists\delta>0,\,[v',v'+\delta]\cap I\subset I_0\] and closed to the right:
\[\exists v'\in I,(v_i)\in I_0^\N,\, v_i\nearrow v'\Rightarrow v'\in I_0,\] which suffices as $I$ is an interval and $u\in I_0$.

For the first part, fix $v'\in I_0\setminus\{v\}$, $n$ and $(u_i)_0^n$, $u_n=v'$ as provided by the definition of $I_0$. By Observation~\ref{obs:diff:monotone} there exists $(v-v')\wedge\pi>\delta>0$ small enough so that $d_{v'}^{\delta}-d_{v'}^+<\varepsilon$, which proves that $[v',v'+\delta]\subset I_0$.

The proof of $I_0$ being closed goes along the same lines looking to the left instead of to the right. More precisely, let $v_i$ form an increasing sequence of elements of $I_0$ converging to $v'\in I$. By definition for $i$ sufficiently large $v'-v_i<\delta$, where $0<\delta<(v'-u)\wedge\pi$ is such that $d_{v'}^{-\delta}-d_{v'}^-<\varepsilon$. Hence, taking a sequence given by the definition of $v_i\in I_0$ and appending $v'$ to it, we obtain $v'\in I_0$, which concludes the proof.
\end{proof}
\begin{rem}
One can use the technique of quasi-stable directions~\cite{Bollobas14} to deal more easily with intervals of unstable and isolated stable directions. We do not do this as our construction works for the more difficult stable intervals and trivially also applies to unstable ones.

Also notice that if one knew that $(u,\theta)\mapsto d_u^\theta$ is continuous, this would follow by uniform continuity on a compact set.
\end{rem}
We shall in fact need the following variant which follows immediately.
\begin{cor}
\label{cor:topo}
With the notation of Lemma~\ref{lem:topo} there also exist two directions such that $v<v'<u'<u$ and
\begin{align*}
d_u^{u'-u}-d_{u}^-&<\varepsilon,\\
d_v^{v'-v}-d_v^+&<\varepsilon.
\end{align*}
\end{cor}
\begin{proof}
Given a sequence as in Lemma~\ref{lem:topo} we apply one step of the reasoning to the right of $v$, obtaining $v'$ sufficiently close to $v$ and one step to the left of $u$. We simply observe that the inequalities we obtained in the proof of the Lemma were in fact the stronger ones in the statement of the corollary.
\end{proof}

\subsection{Critical density characterisation of $\qct$---proof of Theorem~\ref{th:general}}
\label{subsec:main}
In order to prove Theorem~\ref{th:general} we will first need to show that above the maximal critical density in a semi-circle a certain well-chosen big droplet of infection grows indefinitely in that direction with high probability. We thus start by defining our droplets (see Figure~\ref{fig:droplet}).
\begin{defn}
Let $n\ge 3$, $u=u_0<\ldots<u_{n+1}=v$ be directions with $u_n=u_1+\pi$ and $u_n<v<u<u_1$ and let $L$ be in $\R_+$. We then define the \emph{droplet} of size $L$ by
\begin{equation}
D_L=\bigcap_{i=0}^{n+1} \H_{u_i}^L-x_L,\quad
D_{L+}=\bigcap_{L'>L}D_{L'}=\left(\bigcap_{i=1}^{n}\H_{u_i}^{L+}-x_L\right)\cap V_{u,v},
\label{eq:def:droplet}
\end{equation}
where $x_L\in\R^2$ is such that $\<x_L,u\>=\<x_L,v\>=L$, so that droplets are inscribed in $V_{u,v}$.
\end{defn}
It is crucial for the reasoning to follow that all sides of this droplet are of length $\Theta(L)$ for large $L$ when the directions are fixed.

\begin{figure}
\begin{center}
\begin{tikzpicture}[line cap=round,line join=round,>=triangle 45,x=0.3cm,y=0.3cm]
%\clip(-31,-10.1) rectangle (11,10.1);
\draw (-1.27,10)-- (3.33,10);
\draw (8.75,5.94)-- (3.33,10);
\draw (10.23,0.47)-- (8.75,5.94);
\draw (9.52,-3.63)-- (10.23,0.47);
\draw (9.52,-3.63)-- (7.13,-7.48);
\draw (2.59,-10)-- (7.13,-7.48);
\draw (-2.1,-10)-- (2.59,-10);
\draw [->] (0,0) -- (-4.02,-9.15);
\draw [->] (0,0) -- (0,-10);
\draw [->] (0,0) -- (4.86,-8.74);
\draw [->] (0,0) -- (0,10);
\draw [->] (0,0) -- (-2.5,9.68);
\draw [->] (0,0) -- (6,8);
\draw [red, ultra thick] (6,8) -- (3.33,10);
\draw (4.5,9) node[anchor=west] {$h_{n-1}$};
\fill[fill=black,fill opacity=0.3] (2.87,8.6) -- (2.87,10.6) -- (4.87,10.6) -- (4.87,8.6) -- cycle;
%3.87,9.6
\draw (-30,0) node[anchor=south west] {$0$};
\draw (-5.5,-3.5) node[anchor=north west] {$L.u_0$};
\draw (3,-3.5) node[anchor=north west] {$L.u_2$};
\draw (0,-7) node[anchor=north west] {$L.u_1$};
\draw (-6,7) node[anchor=north west] {$L.u_{n+1}$};
\draw (3,5) node[anchor=north west] {$L.u_{n-1}$};
\draw (0,8) node[anchor=north west] {$L.u_n$};
\draw (3.5,0) node[anchor=west] {$\vdots$};
%\draw [domain=-30.468944957100955:11.0] plot(\x,{(--301.55--7.53*\x)/29.2});
\draw (-30.47,2.47) -- (-1.27,10);
\draw (-30.47,2.47) -- (-2.1,-10);
%\draw [domain=-30.468944957100955:11.0] plot(\x,{(-309.88-12.47*\x)/28.37});
\end{tikzpicture}
\end{center}
\caption{The droplet $D_L$ of size $L$ for the directions $u_0,\ldots u_{n+1}$ defined in~\eqref{eq:def:droplet}. The left `half-side', $h_{n-1}$ of $l_{n-1}$ is thickened. The shaded box is $x+B_{L/C}$ for some $x\in h_{n-1}$.}
\label{fig:droplet}
\end{figure}

The growth mechanism is, of course, quite different from the one encountered for critical and supercritical models (finding an infection somewhere on the side of a droplet and relying on quasi-stable directions to make sure that the sides expand to fill the corners as well). Our strategy is to infect sites one by one by inspecting an area of size $\Omega(L)$ to have sufficiently small probability that the site remains uninfected in that zone. We can then use the union bound to infect a new row on one side of the droplet. We use this procedure to make the droplet grow, making sure that each side grows linearly, so that we can finally sum the probabilities using the decay provided by the definition of critical densities.

The next lemma roughly tells us that once a set of directions is fixed as in Corollary~\ref{cor:topo}, a large infected droplet is highly likely to grow to infect the cone it is inscribed in if given a sufficiently high (compared to the critical densities) additional density of infections.
\begin{lem}
\label{lem:general:main}
Let $n>2$ and let $(u_i)_0^{n+1}$ be directions such that \[u=u_0<u_1<\ldots<u_n<u_{n+1}=v,\]
and $u_1+\pi=u_n<u_{n+1}<u_0<u_1$. Fix $C$ large enough depending on the directions. Let $q>\max d_{u_{i-1}}^{u_i-u_{i-1}}$ for all $1\le i\le n+1$ and let $\delta>0$. Then for $L$ large enough and for any $\Lambda\ge CL$
\[\Pr_q\left(\left[D_L\cup (A\cap B_{C\Lambda})\right]\supset V_{u,v}\cap B_{\Lambda/C}\right)>1-\delta.\]
\end{lem}
\begin{proof}
Let $(u_i)_{i=0}^n$, $C$, $q$ and $\delta$ be as in the statement of the lemma.

Consider $L$ such that $\Z^2\cap(D_{L+}\setminus D_L)\neq\varnothing$ and let $L'=\sup\{l,\,D_{l}\cap\Z^2= D_{L+}\cap\Z^2\}$. Consider the (possibly empty) new line of $D_{L'}\setminus D_L$ in direction $u_i$, $l_i=\Z^2\cap D_{L'}\cap\left(\left(\H^{L'}_{u_i}\setminus\H^{L}_{u_i}\right)-x_L\right)$, for $1\le i\le n$. Let $h_i=\{x\in l_i,\,\langle u_i+\pi/2,x+x_L\rangle\ge 0\}$ be the left \emph{half-side} of $l_i$ (looking from inside the droplet), see Figure \ref{fig:droplet}. For each site $x\in h_i$ and $\Lambda\ge CL$ we have
\begin{align*}
\Pr_q\left(x\not\in\left[D_L\cup (A\cap B_{C\Lambda})\right]\right)&\le\Pr_q\left(x\not\in\left[(A\cup D_L)\cap (x+B_{L/C})\right]\right)\\
&\le\Pr_q\left(0\not\in\left[\left(A\cup V_{u_i,u_{i+1}}\right)\cap B_{L/C}\right]\right),\end{align*}
since inside a box of size $L/C$ around $x$ the droplet locally looks like (at least) $V_{u_i,u_{i+1}}$, see Figure \ref{fig:droplet}.
Then the union bound over all sites in all half-sides gives
\begin{multline*}\Pr_q\left(\left[D_L\cup (A\cap B_{C\Lambda})\right]\not\supset D_{L'}\right)\le \sum_{i=1}^n|l_i|\Big(\Pr_q\left(0\not\in\left[\left(A\cup V_{u_i,u_{i+1}}\right)\cap B_{L/C}\right]\right)\\
+\Pr_q\left(0\not\in\left[\left(A\cup V_{u_{i-1},u_i}\right)\cap B_{L/C}\right]\right)\Big).
\end{multline*}

We now iterate this bound. Let $L_0$ be large enough (depending on $C$, $\delta$ and $(u_i)^{n+1}_{i=0}$) and such that such that $\Z^2\cap(D_{L_0+}\setminus D_{L_0})\neq\varnothing$. Define $L_{j+1}=\sup\{l,\,D_{l}\cap\Z^2=D_{L_j+}\cap\Z^2\}$ for all $j\ge0$. Again by the union bound for any $L\ge L_0$ and $\Lambda\ge CL$ we have
\begin{multline*}
\Pr_q\left(\left[D_L\cup (A\cap B_{C\Lambda})\right]\not\supset D_{\Lambda}\right)\leq \sum_{i=1}^{n}\sum_{j=0}^\infty |l_i^j|\Big(\Pr_q\left(0\not\in\left[\left(A\cup V_{u_i,u_{i+1}}\right)\cap B_{L_j/C}\right]\right)\\
+\Pr_q\left(0\not\in\left[\left(A\cup V_{u_{i-1},u_i}\right)\cap B_{L_j/C}\right]\right)\Big),
\end{multline*}
where $l_i^j=\Z^2\cap D_{L_{j+1}}\cap\left(\left(\H_{u_i}^{L_{j+1}}\setminus\H_{u_i}^{L_j}\right)-x_{L_j}\right)$.

Let us upper bound the first term for $i=1$ for concreteness. Let $j_k=\min\{j,\,L_j\ge Ck\}$. Then for any $k\ge \lfloor L_0/C\rfloor$
\[\sum_{j=j_k}^{j_{k+1}-1}|l^j_1|\Pr_q\left(0\not\in\left[\left(A\cup V_{u_1,u_2}\right)\cap B_{L_j/C}\right]\right)\le \Pr_q\left(0\not\in\left[\left(A\cup V_{u_1,u_2}\right)\cap B_k\right]\right)\sum_{j=j_k}^{j_{k+1}-1}|l^j_1|.\]
Finally, the last sum is easily seen to be at most $C^3k$ (it is essentially equal to the area covered by the $u_i$ side while growing from $D_{Ck}$ to $D_{C(k+1)}$), so in total we get
\[
\Pr_q\left(\left[D_L\cup (A\cap B_{C\Lambda})\right]\not\supset D_{\Lambda}\right)\leq \!\!\!\sum_{k=\lfloor L_0/C\rfloor}^\infty\!\!\! C^3k\sum_{i=0}^{n}\Pr_q\left(0\not\in\left[\left(A\cup V_{u_i,u_{i+1}}\right)\cap B_k\right]\right)\le\delta\]
by Definition~\ref{def:diff} and the choice of $q$. This concludes the proof, since $D_\Lambda\supset V_{u,v}\cap B_{\Lambda/2}$ (by construction the $u,v$-sector of the Euclidean ball of radius $\Lambda/C$ is contained in $D_\Lambda$).
\end{proof}

We are now ready to prove Theorem~\ref{th:general}.
\begin{proof}[Proof of Theorem~\ref{th:general}]
By Observation~\ref{obs:diff:monotone} we have
\[\qct\ge \sup_{u\in S^1}d_u\ge \inf_{C\in\cC}\sup_{u\in C}d_u,\]
so we are left with proving $\qct\le \inf_{C\in\cC}\sup_{u\in C}d_u$.

Fix $\varepsilon>0$ sufficiently small and $C\in\cC$ such that
\[\varepsilon+\inf_{C'\in\cC}\sup_{u\in C'}d_u>\sup_{u\in C}d_u.\] Also fix a set of directions as required by Lemma~\ref{lem:general:main} with $C=[u_1,u_n]$ and satisfying 
\begin{align*}
\forall i\in[2,n],\;d_{u_{i-1}}^{u_i-u_{i-1}}-(d_{u_{i-1}}^+\vee d_{u_i}^-)&< \varepsilon\\
d_{u_1}^{-(u_1-u_0)}-d_{u_1}^-&<\varepsilon\\
d_{u_n}^{u_{n+1}-u_n}-d_{u_n}^+&<\varepsilon,
\end{align*}
as provided by Corollary~\ref{cor:topo}. Without loss of generality (after rotating the lattice) we assume $u_n=(0,1)$. Fix $\delta>0$ sufficiently small depending on the directions $(u_i)$ and $\varepsilon$. Let $q'=2\varepsilon+\sup_{u'\in C}d_{u'}$, so that $q=q'-\varepsilon$ satisfies the condition $q>\max d_{u_{i-1}}^{u_i-u_{i-1}}$ of Lemma~\ref{lem:general:main}.

We sample (a part of) the infected sites as the union of two independent percolations---one with probability $\varepsilon$ and another one with probability $q$. At this point one can easily obtain $q'\ge\qc$ using Lemma~\ref{lem:general:main} to prove that a droplet of size $L$ grows with high probability in the second percolation and find such a large droplet in the first one. However, in order to avoid using $\qc=\qct$, we give a slightly more involved but fairly standard renormalisation procedure to prove the desired inequality for $\qct$. Furthermore, we will be able to deduce that $\qct$ is also the critical probability of exponential decay.

Let $L$ be large enough for the assertion of Lemma~\ref{lem:general:main} to hold. Also fix $N$ sufficiently large depending on $L$ such that $\Pr_\varepsilon(\exists x\in B_N,\,A\cap B_N\supset D_L+x)\ge 1-\delta$. Finally, let $c\in \N$ be large enough depending only on the directions $(u_i)$ (and on the constant $C$ in Lemma~\ref{lem:general:main}), but not on $\delta$. Consider a renormalised lattice $\cL=\Z^2$ and say $X\in \cL$ is open if $N.X+B_N\subset[A\cap (N.X+B_{cN})]$. This process is clearly only $2c$-dependent\footnote{Each site is independent from the states of sites at distance more than $2c$ from it.} and we claim that each site is open with probability at least $1-2\delta$. Indeed, $N(X-(\lfloor \sqrt{c}\rfloor,0))+B_N$ contains a droplet of size $L$ in the percolation process with parameter $\varepsilon$ with probability at least $1-\delta$ and this droplet grows to infect $NX+B_{N}$ with probability at least $1-\delta$ in the percolation process with parameter $q$ only using infections inside $NX+B_{cN}$ by Lemma~\ref{lem:general:main}.

Hence, by the Liggett--Schonmann--Stacey theorem \cite{Liggett97} the renormalised process stochastically dominates an independent site percolation with parameter $1-\delta'$ with $\delta'$ which can be made arbitrarily small by choosing $\delta$ sufficiently small. In particular, it is known (from the standard Peierls argument, see e.g.\ \cite{Grimmett99}) that the probability that there is no contour (self-avoiding closed path) of open sites around $0$ decays exponentially. Yet, if such a contour exists in a renormalised box of size $a>c$, we know that $0\in[A\cap B_{2aN}]$. Indeed, since the family is not trivial subcritical, the renormalised site $NX+B_N$ for $X$ in the contour becomes infected using $A\cap (NX+B_{cN})$ and the union of these sets for all $X$ in the contour is enough to infect the origin. To see this, simply use the fact that there exists an unstable direction and that the BP process inside the infected contour behaves as though everything outside the contour is infected. Thus, $\theta_m(q')$ decays exponentially in $m$, since $N$ is a constant. Hence, $q'\ge\qct$, concluding the proof of \eqref{eq:qct:decomposition}.

Let us now consider a non-subcritical family and show that $\qct=0$. Fix $q>2\varepsilon$. It is not hard to see (e.g.\ by repeating the proof from~\cite{Bollobas15}) that a sufficiently large droplet is very likely to grow using a density $\varepsilon$ of infections to infect an entire cone of fixed opening depending only on $\varepsilon$ and $\cU$ (see Figure~7 of~\cite{Bollobas15}). We can then repeat the renormalisation above using this input instead of Lemma~\ref{lem:general:main} to obtain that there is exponential decay at $q$ and thereby $\qct=0$.
\end{proof}
\begin{rem}
\label{rem:exp:decay}
Note that we also proved that $\qct$ is the critical probability of exponential decay: for each $q>\qct$
\[\liminf_n\frac{-\log\theta_n(q)}{n}>0,\]
while this fails for $q<\qct$. Moreover, since the family is not trivial, the exponential decay of the absence of a renormalised contour of radius $n$ implies also exponential decay of $\Pr_q(\tau_0\ge n)$ for $q>\qct$.
\end{rem}
\begin{rem}
In fact, using droplets contained between two parallel lines (see Figures~5 and~7 of \cite{Bollobas15}) instead of a cone with strictly positive opening one can obtain a slightly stronger characterisation of $\qct$ only involving one of the left or right critical densities at each endpoint of the semi-circle.
\end{rem}

\section{Critical densities of oriented percolation}
\label{sec:OP}
In this section we determine the critical densities of the simplest subcritical BP model---OP. This is established in order to be used in conjunction with Theorem~\ref{th:general} in the next section to deduce information about other models. Interestingly, although determining critical densities corresponds to studying the phase transition of OP with an absorbing boundary condition (in a restricted region), this problem does not seem to have been thoroughly studied. The only case which we are aware of that has been considered~\cite{Frojdh01} is the symmetric one---$u=\pi$, for which the result, as we shall see, is that the transition is the same as on the entire plane.

Let us recall a few classical results from OP theory all of which can be found up to minor modifications in Durrett's review~\cite{Durrett84} (see also \cites{Durrett80,Liggett05,Durrett83,Gray80}). We will not redo most of the proofs, as they will be discussed in more detail for GOP in an upcoming work by Szab\'o and the author~\cite{Hartarsky20GOP} and since they have appeared numerous times in the literature in slightly modified forms.

Recall that OP is defined by $\cU=\{U\}=\{\{(-1,1),(1,1)\}\}$. For the sake of convenience, in this section we parametrise in terms of $p=1-q$---the density of healthy (open) sites, so that $\Pr_p$ still denotes the product Bernoulli measure such that each site is open with probability $p$. For the rest of this section we consider only the sublattice of $\Z^2$ generated by $U$ without further mention. Denote by $x\to y$ for $x$ and $y$ in $\Z^2$ the event that there exist $x_0,\ldots,x_N$ with $x_0=x$, $x_N=y$, $x_{i}-x_{i-1}\in U$ and $x_i$ open for $0<i\le N$, that we call an \emph{OP path} (from $x$ to $y$). Let
\[r_n=\sup \left\{x\in\Z,\,\exists y\le 0,\,(y,0)\to(x,n)\right\}\]
be the \emph{right edge} with the convention $\sup\varnothing=-\infty$.
\begin{lem}
\label{lem:OP:alpha}
There exists a function $\alpha:[0,1]\to[-\infty,1]$ called \emph{edge speed} with the following properties.
\begin{enumerate}
\item For any $p$ we have $\Pr_p$-a.s.
\[r_n/n\to \alpha(p)=\inf_n\Ex_p[r_n/n]=\lim_n\Ex_p[r_n/n].\]
\item $\alpha$ is strictly increasing on $\left[\pcop,1\right]$.
\item $\alpha$ and continuous on $\left[\pcop,1\right]$ with $\alpha\left(\pcop\right)=0$, $\alpha(1)=1$ and $\alpha(p)=-\infty$ for $p<\pcop$.
\end{enumerate}
\end{lem}
The first equalities and the a.s.\ limit are proved as in~\cite{Liggett05}, following~\cites{Durrett84,Durrett80}. The other assertions are proved exactly as in~\cite{Durrett84}. We will use this definition of $\alpha$ in the remainder of the paper. The contour argument used in~\cite{Durrett84} to prove the continuity of $\alpha$ (together with the Borel-Cantelli lemma) actually gives the following.
\begin{lem}
\label{lem:OP:path}
For all $p>\pcop$ and $a<\alpha(p)$ we have that with positive probability there exists an infinite OP path $((a_i,i))_{i\in\N}$ with $a_0=0$ and $\inf_n a_n/n\ge a$.
\end{lem}
The next Lemma can be proved exactly like Theorem~7 of~\cite{Griffeath81} (see also~\cite{Durrett84}).
\begin{lem}
\label{lem:OP:exp:decay:supercrit}
If $a>\alpha(p)$, then for some $\gamma>0$
\[\Pr_p(r_n\ge an)\le e^{-\gamma n}.\]
\end{lem}
The following bound on $\alpha$ will only be used in the next section.
\begin{lem}\label{lem:OP:alpha:bound}
For all $p\in[0,1]$ we have
\[\alpha(p)\le \frac{p^3+p-1}{p^3-2p^2+3p-1}.\]
\end{lem}
\begin{proof}
The two-paragraph argument of Section~2 of \cite{Gray80} adapts immediately to give that $\alpha^{-1}(a)$ is larger than the root of the equation
\[(p^3-p^2+2p-1)/(p-p^2)=\frac{1+a}{1-a}.\]
Rephrasing this we obtain exactly the desired inequality.
\end{proof}

Let $\psi$ be the composition of the tangent, the inverse of $\alpha$ and finally $1-\cdot$
\[\psi:[-\pi,-3\pi/4]\cup[-\pi/4,0]\xrightarrow{|\tan|}[0,1]\xrightarrow{\alpha^{-1}}\left[\pcop,1\right]\xrightarrow{1-\cdot}[0,\qc].\]

Putting the preceding facts together we obtain the critical densities of OP.
\begin{thm}\label{th:diff:OP}
The critical density of $\cU=\{U\}=\{\{(1,1),(-1,1)\}\}$ is given by
\[d_u(\cU)=\begin{cases} 0, & u\in[-3\pi/4,-\pi/4]\\
1-\pcop=\qc, & u\in[0,\pi]\\
\psi(u),&\text{otherwise}.
\end{cases}\]
For bidirectional OP $\cU'=\{U,-U\}$, where $-U=\{(-1,-1),(1,-1)\}$, the critical densities are $d_u(\cU')=d_u(\cU)\wedge d_{-u}(\cU)$. One also has $d_u^0=d_u^\pm=d_u$ for all $u$ in both cases.
\end{thm}
\begin{rem}
\label{rem:diff:OP}
If the OP rule is rather $\tilde U=\{(x,y),(z,t)\}$ with the two linearly independent vectors (sites), let $L\in GL_2(\R)$ be such that $L\cdot \tilde U=U=\{(-1,1),(1,1)\}$ and $\det L>0$. Then the critical densities are also transformed via $d_u^{\{\tilde U\}}=d_{u'}^{\{U\}}$, where $u'$ is the direction of $(L(u-\pi/2))+\pi/2$.
\end{rem}
\begin{proof}[Proof of Theorem~\ref{th:diff:OP}]
If $u\in(-3\pi/4,-\pi/4)$ we have nothing to prove, as the directions are unstable. By symmetry it suffices to treat $u\in[-\pi/4,\pi/2]$, so fix one such direction and let $\tilde q=\qc$ if $u\in(0,\pi/2]$ and $\psi(u)$ otherwise. Notice that $\alpha(1-\tilde q)=-\tan(u)$ in the latter case and $0$ in the former one.

Let $q<\tilde q$. By Lemmas~\ref{lem:OP:alpha} and~\ref{lem:OP:path} we know that with positive probability there exists an infinite OP path of healthy sites starting at $0$ not intersecting $\H_u$. This proves that $q\le d_u^\theta$ for all $\theta$, so $q\le d^0_u\le d_u^{\pm}\le d_u$ and the same inequalities hold for $\tilde q$.

Conversely, let $q>\tilde q$. Then by Lemmas~\ref{lem:OP:alpha} and~\ref{lem:OP:exp:decay:supercrit}
\[\Pr_q(0\not\in[(A\cap B_n)\cup V_{u-\theta,u+\theta}])\]
decays exponentially for $\theta>0$ small enough, so that $d_u^0\le d_u^{\pm}\le d_u\le q$. Thus, with the inequalities from the previous case we obtain
\[d_u=d_u^{\pm}=d_u^{0}=\tilde q.\]

Now consider bidirectional OP. It is clear that $0$ remaining healthy for this process is equivalent to $0$ remaining healthy for the family $\{U\}$ and for the family $\{-U\}$, both of which are simply OP. Moreover, these two events are independent conditionally on the state of $0$ (as the oriented paths occur in the upper and lower half-planes respectively). Thus, the critical densities are indeed obtained as claimed.
\end{proof}

\begin{rem}
In order to be able to usefully apply Corollary~\ref{cor:upper:bound} in full generality to any subcritical model, we require a generalisation of Theorem~\ref{th:diff:OP} to GOP. Indeed, every non-trivial subcritical model contains rules corresponding to GOP as explained in Section \ref{subsec:models}. The proof of Theorem~\ref{th:diff:OP} remains unchanged for GOP, provided we have all the ingredients needed, Lemmas \ref{lem:OP:alpha}--\ref{lem:OP:exp:decay:supercrit}. In an upcoming work Szab\'o and the author~\cite{Hartarsky20GOP} explain how those are established.
\end{rem}

\section{Applications of the upper bound for bootstrap percolation}
\label{sec:app}
The most natural and easy way to use Corollary~\ref{cor:upper:bound}, which we call \emph{basic bound}, is for subfamilies consisting of only one rule:
\begin{equation}
\label{eq:basic:bound}
\qc(\cU)\le\qct(\cU)\le\inf_{C\in\cC}\sup\min_{U\in\cU}d_u(\{U\}),
\end{equation}
since the r.h.s.\ terms correspond to OP treated in the previous paragraph or similarly behaved GOP. In principle this approach includes the trivial one consisting of using $\qcu\le\min_{U\in\cU}q_{\mathrm{c}}^{\{U\}}$, but also allows better estimates. 

We give two illustrative applications of the general bound of Corollary~\ref{cor:upper:bound}. The first one follows from the basic bound given by single rule subfamilies as outlined above, while the second one is more subtle.

\subsection{The basic bound---the DTBP model}
Our first example is DTBP. We improve the upper bound of~\cite{Balister16} as asked in their Question~17 by proving Theorem~\ref{th:BBPS:family}.
\begin{figure}
\begin{center}
\begin{tikzpicture}[line cap=round,line join=round,>=triangle 45,x=12.0cm,y=4.0cm]
\draw[->,color=black] (0,0) -- (1,0);
\draw[->,color=black] (0,0) -- (0,1);
\draw[color=black] (1,-10pt) node[left] {\footnotesize $5\pi/4$};
\draw[color=black] (0.75,-10pt) node {\footnotesize $\pi$};
\draw[color=black] (0.5,-10pt) node {\footnotesize $3\pi/4$};
\draw[color=black] (0.25,-10pt) node {\footnotesize $\pi/2$};
\draw[color=black] (0,-10pt) node[right] {\footnotesize $\pi/4$};
\draw[color=black] (0,0) node[left] {\footnotesize $0$};
\draw[color=black] (0,.68) node[left] {\footnotesize $1-\alpha^{-1}(1/3)$};
\draw[color=black] (0,1) node[left] {\footnotesize $1-\pcop$};
\clip(0,-0.05) rectangle (1,1.05);
\draw [line width=1pt] (0,1)-- (0.5,1);
\draw [line width=1pt] (0.75,0)-- (1,0);
\draw [line width=1pt,dotted,color=ffqqqq] (0.25,0)-- (0.5,0);
\draw [shift={(-1.645,-0.325)},line width=1pt,dotted,color=ffqqqq]  plot[domain=0.17:0.76,variable=\t]({1*1.92*cos(\t r)+0*1.92*sin(\t r)},{0*1.92*cos(\t r)+1*1.92*sin(\t r)});
\draw [shift={(1.65,-0.32)},line width=1pt,dash pattern=on 3pt off 3pt,color=qqffqq]  plot[domain=2.38:2.97,variable=\t]({1*1.92*cos(\t r)+0*1.92*sin(\t r)},{0*1.92*cos(\t r)+1*1.92*sin(\t r)});
\draw [shift={(-3.16,-0.45)},line width=1pt]  plot[domain=0.11:0.38,variable=\t]({1*3.94*cos(\t r)+0*3.94*sin(\t r)},{0*3.94*cos(\t r)+1*3.94*sin(\t r)});
\draw [shift={(2.78,-0.04)},line width=1pt,dotted,color=ffqqqq]  plot[domain=2.67:3.12,variable=\t]({1*2.28*cos(\t r)+0*2.28*sin(\t r)},{0*2.28*cos(\t r)+1*2.28*sin(\t r)});
\draw [line width=1pt,dotted,color=ffqqqq] (0.75,1)-- (1,1);
\draw [line width=1pt,dash pattern=on 3pt off 3pt,color=qqffqq] (0.25,1)-- (1,1);
\end{tikzpicture}
\end{center}
\caption{A schematic representation of the critical densities of the three OP rules in DTBP. For symmetry reasons we only depict the domain $u\in[\pi/4,5\pi/4]$.}
\label{fig:DTBP}
\end{figure}
\begin{proof}[Proof of Theorem~\ref{th:BBPS:family}]
Our starting point is \eqref{eq:basic:bound}. Let $U_i$ be the three rules in the update family $\cU$ of DTBP defined in \eqref{eq:def:DTBP}. We can then use Theorem~\ref{th:diff:OP} and Remark~\ref{rem:diff:OP} to determine the r.h.s. We spare the reader the tedious details, but it is elementary to see (see Figure~\ref{fig:DTBP}) that by symmetry there are three local maxima of $u\mapsto\min_i d_u(\{U_i\})$---the one at $\pi/4$ being the global maximum in $[-\pi/4,3\pi/4]$. Hence, Theorem~\ref{th:diff:OP} and Remark~\ref{rem:diff:OP} give
\[\qcu\le d^{\mathrm{OP}}_{L(-\pi/4)+\pi/4}=d^{\mathrm OP}_{\arctan(-1/3)}=1-\alpha^{-1}(1/3),\]
where $L(x,y)=(x,y-x)$ transforms the DTBP rule $\{(-1,-1),(0,1)\}$ into $\{(-1,0),(0,1)\}$, which is OP rotated by $\pi/4$.

In fact, the other two maxima are also easily determined to be at $\pi-\arctan(1/2)$ and $\arctan(1/2)-\pi/2$. They turn out to give the same value as the one at $\pi/4$, but we did not need that for establishing the upper bound. Finally, Lemma~\ref{lem:OP:alpha:bound} provides the desired bound $\alpha^{-1}(1/3)>0.7548$.
\end{proof}
It should be noted that the numerical bound is not optimised, but merely given to testify that the gain is significant. For comparison, based on a refinement of the same method in~\cite{Gray80} in conjunction with the trivial bound $\qcu\le 1-\pcop=1-\alpha^{-1}(0)$ the authors of~\cite{Balister16} obtain $\qcu<0.312$. Even if the exact value of $\pcop$ were known, it follows from rigorous upper bounds that the trivial bound cannot go beyond $0.274$~\cite{Balister93}. Numerical studies indicate that in fact $1-\pcop\approx 0.2945$~\cite{Onody92}. Unfortunately, we have been unable to find appropriate numerical estimates for $\alpha$ for values far from $\qc$ in the literature, so we cannot provide a corresponding result for our bound $1-\alpha^{-1}(1/3)$. Finally, all these values are also to be compared with the numerical estimate $\qcu\approx 0.118$ suggested in~\cite{Balister16}, which indicates that there is much room for further improvements.

\subsection{Motivation of the second-level bound}
Unfortunately, the basic bound \eqref{eq:basic:bound} is not tight. Something more, it is possible to find two rules $U_1$ and $U_2$, such that $d(\{U_1,U_2\})$ is nowhere equal to $d(\{U_1\})\wedge d(\{U_2\})$. Even worse, changing $U_2$ may lead to a change in $d(\{U_1,U_2\})$ while $d(\{U_2\})$ remains the same. We give the following instructive counterexample, along whose lines many can be constructed.

\begin{prop}
\label{prop:counterexample}
Let $\cU_n=\{U_1,U_n\}=\{\{(1,1),(-1,1)\},\{(n,n),(-n,n)\}\}$ for $n\in\N$. Then as $n\to\infty$
\[q_{\mathrm{c}}(\cU_n)\le 1-\inf\left\{p,\,\pcop\le \theta^{\mathrm OP}(p)\right\}+o(1),\]
where $\theta^{\mathrm OP}(p)=\Pr_{1-p}\left(0\not\in[A]_{\{U_1\}}\right)$ is the probability that $0$ is never infected in OP.
\end{prop}
\begin{proof}
Let $B'_{n}=(-n,n]\times(0,n)$ and denote by $\cL=\{n.(m-k,m+k),\;m,k\in\N\}$ the sites concerned by the second rule. Note that for all $x\in\cL$ the boxes $x+B'_n$ are disjoint and disjoint from $\cL$.

Fix $\varepsilon>0$ and $p=1-q$ such that $\theta^{\mathrm OP}(p)<\pcop-\varepsilon$. Let $n$ be large enough so that \[\Pr_q\left(x\not\in\left[A\cap \left(x+B'_n\right)\right]\right)\le \frac{\theta^{\mathrm OP}(p)+\varepsilon}{p}.\] Such an $n$ exists, because the process with initial infection in $x+B'_n$ is identical to the one under the family $\{U_1\}$, which is OP and for which we know that the probability converges to $\theta^{\mathrm OP}(p)/p$.

Then we can associate to each site of $x\in\cL$ an independent Bernoulli($\theta^{\mathrm OP}(p)+\varepsilon$) random variable---the indicator of the event $G_x=\{x\not\in A;\,x\not \in [A\cap (x+B'_n)]\}$. Furthermore, $\{x\not\in[A]\}\subset G_x$ for all $x$. But then in order for $0$ to remain uninfected at all times it is necessary to have an infinite path with steps in $U_n$ starting at $0$ of sites $x$ such that $G_x$ occurs and the probability of this event is $\theta^{\mathrm OP}(\theta^{\mathrm OP}(p)+\varepsilon)=0$, since $\theta^{\mathrm OP}(p)\le \pcop-\varepsilon$.
\end{proof}

This example shows where the main difficulty of the subcritical models resides once GOP is well understood. The division into three universality classes is based on the unstable directions of a model, which can be directly obtained by superimposing the ones for each rule, which are very easy to determine~\cites{Bollobas15,Balister16}. In the refined result based on `difficulties' for critical models~\cite{Bollobas14} Bollobás, Duminil-Copin, Morris and Smith only require information in the finitely many isolated stable directions---their difficulty. In their case, like here, there is no easy way of calculating the difficulty of an isolated stable direction without looking at the entire update family. However, in the simple case of critical models the difficulty happens to be a finite discrete quantity, which invites direct exhaustive computation (which for simple models is readily done by hand), and indeed~\cite{Bollobas14} does not provide a recipe for determining difficulties (it turns out that determining them is NP-hard \cite{Hartarsky18NP}). This is essentially the same problem that we are facing here, but the critical densities of subcritical models being much richer, they are even harder to decompose and analyse.

On the bright side the bound from Corollary~\ref{cor:upper:bound} need not be applied to subfamilies with a single rule. Hence, if we have information on the joint critical densities of, say, all pairs of rules in the family $\cU$, then we can extract a (better) upper bound for $\qcu$. We next turn our attention to an example where this approach works brilliantly, while to apply the basic bound (and obtain worse results) we would need an understanding of GOP.

\subsection{Spiral model}
Indeed, in the Spiral model the subfamilies with two rules happen to be simpler than the single-rule ones when restricted to appropriate half-planes. Recall the definition of its update family $\cU=\{U_1,U_2,U_3,U_4\}$ from \eqref{eq:def:Spiral}. We will use Corollary~\ref{cor:upper:bound} to provide a new proof of one of the main results of \cite{Toninelli08}---Theorem \ref{th:Spiral}.

The proof is nearly complete at this point, but we need one last ingredient---a variant of Lemma~4.11 of~\cite{Toninelli08}, which is actually more naturally expressed in the language of critical densities. This is where one uses the ``no parallel crossing'' property, which Jeng and Schwarz~\cite{Jeng08} identified as essential, as without it the pairs of rules do not simplify to OP.
\begin{lem}[Adaptation of Lemma~4.11 of~\cite{Toninelli08}]
\label{lem:Spiral:TB:nocross}
Let $u\in(\pi/2,5\pi/4)$. Then
\[d_u(\{U_1,U_2\})=d_u(\cU'),\]
where $\cU'=\{\{(0,1),(1,1)\},\{(0,-1),(-1,-1)\}\}$ is a bidirectional OP.
\end{lem}
Since there are a few additional technicalities, we give the proof, focusing on the new parts, so the reader is also invited to consult~\cite{Toninelli08} for more details.
\begin{figure}
\begin{center}
\begin{tikzpicture}[line cap=round,line join=round,>=triangle 45,x=0.25cm,y=0.25cm]
\clip(0,-21) rectangle (22,0);
\fill[fill=black,fill opacity=0.3] (24,2.) -- (1,-21) -- (24,-21) -- cycle;
\draw [line width=1.2pt,color=ffqqqq] (1,-12)-- (1,-20.94);
\draw [line width=1.2pt,color=ffqqqq] (2,-11)-- (1,-12);
\draw [line width=1.2pt,color=ffqqqq] (2,-9)-- (2,-11);
\draw [line width=1.2pt,color=ffqqqq] (3,-8)-- (3,-2);
\draw [line width=1.2pt,color=ffqqqq] (5,-2)-- (4,-1);
\draw [line width=1.2pt,color=ffqqqq] (2,-9)-- (3,-8);
\draw [line width=1.2pt,color=ffqqqq] (4,-1)-- (3,-2);
\draw [line width=1.2pt,color=ffqqqq] (5,-2)-- (6,-2);
\draw [line width=1.2pt,color=ffqqqq] (7,-3)-- (6,-2);
\draw [line width=1.2pt,color=ffqqqq] (6,-4)-- (7,-3);
\draw [line width=1.2pt,color=ffqqqq] (6,-5)-- (6,-4);
\draw [line width=1.2pt,color=ffqqqq] (5,-6)-- (6,-5);
\draw [line width=1.2pt,color=ffqqqq] (6,-7)-- (5,-6);
\draw [line width=1.2pt,color=ffqqqq] (7,-7)-- (6,-7);
\draw [line width=1.2pt,color=ffqqqq] (8,-8)-- (7,-7);
\draw [line width=1.2pt,color=ffqqqq] (8,-9)-- (7,-10);
\draw [line width=1.2pt,color=ffqqqq] (7,-10)-- (7,-11);
\draw [line width=1.2pt,color=ffqqqq] (8,-13)-- (9,-12);
\draw [line width=1.2pt,color=ffqqqq] (9,-12)-- (8,-11);
\draw [line width=1.2pt,color=ffqqqq] (8,-11)-- (7,-11);
\draw (24,2)-- (1,-21);
\draw [line width=1.2pt,color=ffqqqq] (8,-9)-- (8,-8);
\end{tikzpicture}
\end{center}
\caption{An example of the healthy path used in the proof of Lemma~\ref{lem:Spiral:TB:nocross}. The shaded region is entirely infected.}
\label{fig:TB}
\end{figure}
\begin{proof}[Proof of Lemma~\ref{lem:Spiral:TB:nocross}]
Let $u\in I=(\pi/2,5\pi/4)$ and $\pi/2-u<\theta<5\pi/4-u$. We claim that $d^\theta_u(\{U_1,U_2\})=d^\theta_u(\cU')$, which clearly implies the desired result. Let $B=[-n,n]\times[0,cn]$ for some fixed $n\in\N$ sufficiently large and $0\le c\le1$ sufficiently small ($c<\tan(u-\pi/2)$ if $u\in(\pi/2,\pi)$ and the same with $u$ replaced with $u+\theta$) and define the events
\begin{align*}
\cE_1&=\left\{0\not\in\left[\left(A\cup V_{u,u+\theta}\right)\cap B\right]_{\cU'}\right\}\\
\cE_2&=\left\{0\not\in\left[\left(A\cup V_{u,u+\theta}\right)\cap B\right]_{\{U_1,U_2\}}\right\}.
\end{align*}
We argue that $\cE_1\supset\cE_2$. Fix a realisation of $A$ such that $\cE_2\setminus\cE_1$ holds and call the sites in
\[B\setminus\left[\left(A\cup V_{u,u+\theta}\right)\cap B\right]_{\{U_1,U_2\}}\]
\emph{survivors}. Consider the rightmost path $P$ of survivors starting at $0$ with steps in $\{(0,1),(1,1)\}$ (performing the step $(1,1)$ whenever possible and $(0,1)$ only when $(1,1)$ is not possible) and denote $x$ its endpoint. Indeed, $P$ cannot reach the (top) boundary $\partial B$ of $B$, since $\cE_1$ does not hold (survivors are necessarily initially healthy). Since $x$ is a survivor and both $x+(0,1)$ and $x+(1,1)$ are not (otherwise $x$ is not the end of the path), there needs to be a survivor $y$ among $x+(1,0)$ and $x+(1,-1)$ (see Figure~\ref{fig:TB}). In particular, $x\neq 0$, as both $(0,1)$ and $(1,-1)$ are in $\H_u\cap\H_{u+\theta}$.

Since $y$ is a survivor, there has to exist a path of survivors starting at $y$ with steps in $U_2$ reaching $\partial B$. However, it is easy to see (see Figure \ref{fig:TB}) that such a path cannot reach $\partial B$ without intersecting $V_{u,u+\theta}$ or $P$. The former possibility is excluded, since $V_{u,u+\theta}$ are not survivors and the latter one contradicts the choice of $P$ to be the rightmost path of survivors from $0$.

Hence, $\cE_2\subset\cE_1$. A similar reasoning applies with $B$ tilted by $3\pi/4$. Finally, recalling that the region $V_{\pi/2,5\pi/4}$ is entirely infected for all values of $(u,\theta)$ considered, we obtain that
\[0\not\in\left[\left(A\cup V_{u,u+\theta}\right)\cap B_n\right]_{\{U_1,U_2\}}\Longrightarrow0\not\in\left[\left(A\cup V_{u,u+\theta}\right)\cap B_{cn}\right]_{\cU'}.\]
The same implication with $\cU'$ and $\{U_1,U_2\}$ swapped is clear from the fact that $U_1\supset \{(0,1),(1,1)\}$ and $U_2\supset\{(0,-1),(-1,-1)\}$, so we are done by Definition \ref{def:diff}.
\end{proof}
\begin{proof}[Proof of Theorem~\ref{th:Spiral}]
First note that if $q<1-\pcop$, then with probability $1$ there exists a bidirectional $\cU'$ path of healthy sites, which remains healthy also for $\cU$. Therefore, $\qc(\cU)\ge1-\pcop$.

We apply Corollary \ref{cor:upper:bound} to $\cU$ and the families $\cU_1=\{U_1,U_2\}$, $\cU_2=\{U_2,U_3\}$, $\cU_3=\{U_3,U_4\}$ and $\cU_4=\{U_4,U_1\}$. We simply bound $d_u(\cU_1)$ by $1$ for $u\in(-\pi,\pi/2]$ and apply Lemma~\ref{lem:Spiral:TB:nocross} and Theorem~\ref{th:diff:OP} with Remark~\ref{rem:diff:OP} to obtain a bound on $d_u(\cU_1)$ for all $u$. By symmetry the same applies to the other three families up to rotation by $\pi/2$. Hence,
\[\qcu\le\qct(\cU)=\sup_{u\in S^1} d_u\le\!\!\!\sup_{u\in(\pi/2,\pi]}\!\!\!d_u(\cU_1)=\!\!\!\sup_{u\in(\pi/2,\pi]}\!\!\!d_u(\cU')\le\sup_{u\in S^1}d_u(\cU')=1-\pcop.\qedhere\]
\end{proof}

\begin{rem}
It is important to note that Lemma~\ref{lem:Spiral:TB:nocross} does not hold for all directions $u$. It is clear, for example, that when $u=0$ it suffices to have an infinite uni-directional healthy path with steps $\{(1,0),(1,-1)\}$ starting at $0$, which occurs for $q<1-\pcop\neq0= d_{u}(\cU')$. Moreover, the complete Spiral model is \emph{not} equivalent to any (uni- or bi-directional) OP, as it is clear from the fact that it has a discontinuous phase transition~\cite{Toninelli08}, while the phase transition of OP is continuous~\cite{Bezuidenhout90}---BP occurs for both bidirectional OP involved, but not for Spiral. Thus, it is crucial to restrict the process to half-planes where it is equivalent to OP. This idea also underlies the reasoning of~\cite{Toninelli08}.
\end{rem}

\section{Exponential decay and applications}
\label{sec:exp}
In Section \ref{sec:critdens} we characterised $\qct$ in terms of critical densities and proved that it is the critical probability of exponential decay. We now give a second proof of the latter, which makes the conclusions slightly stronger and more manipulable. For instance, if we assume that $\theta_n(q)$ decays like a power law, \eqref{eq:def:pct} gives that for $q<\qct$ the exponent is at least $-2$, which is what we will prove here without assuming that the decay is a power law. Moreover, this method will grant us access to noise sensitivity as well as proving that a one-arm event has strictly positive probability below $\qct$, so that this is indeed a phase transition regardless of whether $\qc=\qct$ or not. Finally, we give a straightforward but important application of exponential decay to the spectral gap and mean infection time of KCM.

As a motivation we start by answering Questions~12 and~14 of Balister, Bollobás, Przykucki and Smith~\cite{Balister16}. We then reprove exponential decay and all the results gathered in Theorem~\ref{th:exp:decay} using the method developed by Duminil-Copin, Raoufi and Tassion~\cite{Duminil-Copin19} and then use a modification of the algorithm we made for the proof of exponential decay to also deduce the results concerning noise sensitivity in Theorem~\ref{th:noise}.
\subsection{Answers to Questions~12 and~14 of~\cite{Balister16}}
Let us begin this section by explaining why, contrary to the expectations of the authors of~\cite{Balister16}, one should expect exponential decay \emph{above} criticality rather than below it, thus answering Question~12 of that paper. As the reasoning will be identical, we also answer Question~14, but before that we will need to establish the following straightforward fact that will serve as a source of examples.
\begin{prop}
\label{prop:GOP:low:pc}
For every $\varepsilon>0$ there exists a GOP model with $\qc\ge1-\varepsilon$.
\end{prop}
\begin{proof}
Fix $1-q=\varepsilon>0$ and let $N=N(\varepsilon)\in\N$ be large enough. Consider the following GOP update family
\[\cU=\{U\}=\{\H_{-\pi/2}\cap B_{8N}\}.\]

We perform the following renormalisation. We call a renormalised site $X\in\Z^2$ good if there is a healthy site in $4N.X+B_N$. The renormalised process clearly yields a percolation with parameter larger than $\pcop$ for $N$ large enough. Indeed, sites are good independently (as $(4NX+B_N)\cap(4NY+B_N)=\varnothing$ for $X\neq Y\in\Z^2$) with probability $1-q^{|B_N|}$. In particular, for $N$ large enough there is a positive probability that the renormalised site $0$ belongs to an infinite OP path of good renormalised sites. But this implies that the ordinary site $0$ belongs to an infinite oriented path of healthy vertices in the graph structure on $\Z^2$ defined by $U$, i.e.\ $0$ remains healthy forever with positive probability. Hence, BP does not occur a.s.\ and $1-\varepsilon=q\le\qc$ as desired.
\end{proof}

\subsubsection{Question~14}
The authors of \cite{Balister16} ask for which subcritical models below criticality there is no infinite path (non-oriented with nearest neighbour steps) of sites in $[A]$ and seem to be in favour of a positive answer for all subcritical BP models. On the one hand, it is indeed \emph{possible} for this scenario to occur and that is the case for the simplest subcritical model---OP. 
\begin{prop}
Consider OP and let $q<\qc$. Then a.s.\ there is no infinite path in $[A]$.
\end{prop}
\begin{proof}
Let $q<\qc$. Recall that the edge speed from Lemma \ref{lem:OP:alpha} satisfies $\alpha(1-q)>\varepsilon$ for some $\varepsilon>0$. It then follows from Lemma \ref{lem:OP:path} that with positive probability there exists an infinite initially healthy oriented path $(a_i,i)_{i\in\N}$ (i.e.\ with $|a_{i+1}-a_i|=1$ for all $i$) starting at $0$ with $\inf a_i/i\ge \varepsilon$. Reflecting this event, we see that with positive probability there exists a bi-infinite oriented path $(a_i,i)_{i\in\Z}$ containing $0$ such that $\inf_{i\neq 0} a_i/|i|\ge\varepsilon$. By ergodicity and symmetry a.s.\ there exist two bi-infinite oriented paths of initially healthy vertices $(a_i)_{i\in\Z}$ and $(b_i)_{i\in\Z}$ such that $a_0<0$, $b_0>0$, $\liminf_{|i|\to\infty}a_i/|i|\ge\varepsilon$ and $\limsup_{|i|\to\infty}b_i/|i|\le -\varepsilon$. As these are oriented paths of healthy sites, they never become infected in the BP process. Moreover, the two paths intersect both in the upper and lower half-planes, $\H_{-\pi/2}$ and $\H_{\pi/2}$, forming a contour of sites in $\Z^2\setminus[A]$ around the origin. In particular, a.s.\ there is no infinite non-oriented path with nearest neighbour steps in $[A]$ containing the origin, which concludes the proof by ergodicity.
\end{proof}

On the other hand, it is obvious that \emph{any} subcritical model with $\qc>\pcsp$ is an example of the opposite behaviour. Minimal such examples are provided by large enough GOP as in Proposition~\ref{prop:GOP:low:pc}, but also by any trivial subcritical model. Indeed, for any $\pcsp<q<\qc$ we a.s.\ have an infinite non-oriented path of initially infected sites.

As we do not give the characterisation asked for in \cite{Balister16}, let us explain why we believe the question to be somewhat extrinsic in the light of the above example and counter-examples. Indeed, the graph structure of $\Z^2$, which defines the infinite path in $[A]$ that \cite{Balister16} asks for, is not relevant to the model itself, defined only by $\cU$. For example if one is to replace $\cU$ by $2\cU$ (e.g.\ in the above examples) the problem is changed non-trivially, while the bootstrap process is really the same. Finally, let us note that we do not expect that $\qc>\pcsp$ (or $\qc\ge\pcsp$) is a necessary condition.

\subsubsection{Question~12}
With the previous reasoning in mind, let us go back to Question~12 of \cite{Balister16} about exponential decay. The question is whether at $q<\qc$ there would be exponential decay in $n$ of the probability of $0$ being connected by sites in $[A]$ to the boundary of $B_n$, to quote \cite{Balister16} ``Here we mean `connected' in the site percolation sense, although other notions of connectedness are also interesting''. 

This is not the case, since in many models there is even no decay at all (the probability of being connected in the non-oriented nearest neighbour sense by sites in $[A]$ to the boundary of $B_n$ may remain bounded away from $0$ as $n\to\infty$ for some $q<\qc$), let alone exponential one. For example consider any subcritical model with $\qc>\pcsp$. Obviously, for $\pcop<q<\qc$ there is positive probability for $0$ to be initially connected to infinity by an infected non-oriented nearest neighbour path, but also with probability $1$ BP does not occur, so some (positive density of) sites remain healthy forever. This is by no means contradictory, since, e.g.\ in the example of Proposition~\ref{prop:GOP:low:pc}, a path, in the graph sense given by the GOP rule and not the non-oriented nearest neighbour one, of healthy sites witnessing that $0$ never becomes infected can easily jump over an infinite infected non-oriented nearest neighbour path in the usual $\Z^2$ sense.

\subsection{Exponential decay---proof of Theorem~\ref{th:exp:decay}}
\label{subsec:exp}
Even though exponential decay below $\qc$ is not always present, we prove that there is exponential decay \emph{above} $\qc$, as it is well known to be the case for OP (this follows e.g.\ from Lemmas \ref{lem:OP:alpha} and \ref{lem:OP:exp:decay:supercrit}). We shall use the recent method of Duminil-Copin, Raoufi and Tassion~\cite{Duminil-Copin19} in order to prove the exponential decay of the one-arm events $E_n$ from Definition \ref{def:En}. In fact, much of the proof of~\cite{Duminil-Copin19} calls for no modification.\footnote{We encourage the reader unfamiliar with that paper to see the second half of the course recording~\cite{Duminil-Copin19course}, which gives precisely the part we need and precisely in the simpler form we use here adapted to product measures, except for Lemma~\ref{lem:exp:decay} we prove.} We will only need the following replacement for their Lemma~3.2.
\begin{lem}
\label{lem:exp:decay}
There exists a randomised algorithm determining $\1_{E_n}$ with maximal revealment
\[\delta\le \frac{3}{n-1}\sum_{k=0}^{n-1}\tilde\theta_k(p).\]
\end{lem}
\begin{proof}
The algorithm is as follows. First pick $k$ uniformly at random in $[1,n-1]$. Let $S\subset B_n$ denote the current set of sites whose state has been checked by the algorithm. We start by revealing (in an arbitrary order) all sites at distance at most $C$ from $\partial B_k$, the boundary of $B_k$, and adding them to $S$. Afterwards we repeat the following. As long as there exists a site $x_0\in B_n\setminus S$ for which there exist an integer $N\ge 1$ and a sequence $x_1,\ldots x_N$ of sites in $S$ verifying the following conditions, the algorithm picks one of the possible $x_0$ arbitrarily and checks its state.
\begin{itemize}
\item $x_N$ is at distance at most $C$ from $\partial B_k$.
\item For all $0<i\le N$ we have $x_{i-1}\in x_i+X$.
\item For all $0<i\le N$ we have that $S$ is a witness of the event $\tau^{B_n}_{x_i}\ge i$.
\end{itemize}
When no such sites remain, the first stage of the algorithm terminates.

If at this point $0\not\in S$, then the algorithm stops. Otherwise, we directly reveal all remaining sites in $B_n$ (in an arbitrary order) and stop.

It is clear that this algorithm does determine $\1_{E_n}$. Indeed, if all sites were revealed, this is vacuously true for any function, while if at the end of the first stage we had $0\not\in S$, we know that $E_k$ does not occur (by definition) and therefore neither does $E_n\subset E_k$ (by extraction of a shorter path from a longer one).

We now proceed to bound its revealment. Fix the value of $k$ and consider a site $x\in \partial B_l$ for some $0\le l\le n$. The events $E_n$ are such that when $x$ is revealed, we are certain that either $E_{|k-l|}$ translated by $x$ occurs or the original event $E_k$ occurs. Hence, its revealment is at most
$\tilde\theta_{|k-l|}(q)+\tilde\theta_{k}(q)$. Taking the average on $k$ this gives a maximal revealment bounded by
\[\frac{3}{n-1}\sum_0^{n-1} \tilde\theta_{l}(p).\qedhere\]
\end{proof}
With this Lemma we are ready to apply the method of~\cite{Duminil-Copin19} to prove Theorem~\ref{th:exp:decay}.
\begin{proof}[Proof of Theorem~\ref{th:exp:decay}]
Let us start by proving the theorem for subcritical models. For the first two items, using Lemma~3.1 of~\cite{Duminil-Copin19} we can repeat the proof of their Theorem~1.2, using the result of \cite{ODonnell05} (instead of its more general form, Theorem~1.1 of~\cite{Duminil-Copin19}) together with our replacement for their Lemma~3.2---Lemma~\ref{lem:exp:decay}---and Russo's formula. Setting 
\[\hat\qc=\sup\left\{q,\,\limsup\frac{\log \sum_0^{n-1}\tilde\theta_k(q)}{\log n}\ge 1\right\},\]
this yields the following.
\begin{itemize}
\item If $q>\hat\qc$, then there exists $c(q)>0$ such that
\[\tilde\theta_n(q)\le \exp(-c(q).n).\]
\item There exists $c>0$ such that for $q<\hat\qc$
\[\tilde\theta(q)\ge c.(\hat\qc-q)>0.\]
\end{itemize}
We next prove that $\hat\qc=\qct$.

First notice that $0\not\in[A\cap B_n]$ implies the existence of a path, in the sense of Definition \ref{def:En}, of sites $x_i$ with $\tau^{B_n}_{x_i}=\infty$ from $0$ to $\partial B_n$ (since there are no finite stable healthy sets) with $x_{i+1}\in x_{i}+X$ and $x_0=0$. But such a path needs to come at distance less than $C/4$ of $\partial B_{n/2}$ at some point $x_k$, so $E_{n/3}$ translated by $x_k$ occurs. Thus, by the union bound
\[\theta_n(q)\le Cn\tilde\theta_{n/3}(q).\]
Therefore, exponential decay for $\tilde\theta_n$ implies exponential decay for $\theta_n$ and thereby $\qct\le \hat\qc$ and for $q>\hat \qc$ we have (for some other $c(q)$)
\[\theta_n(q)\le \exp(-c(q).n).\]

Conversely, we know that for $q<\hat\qc$ the sequence $\tilde\theta_n(q)$ converges to $\tilde\theta(q)>0$. Note that on the event $E_n$ there exists a site $x$ with $\tau^{B_n}_x\ge n/C$ at distance at most $C/4$ from $\partial B_{n/2}$ in the path in Definition \ref{def:En}. Then by the union bound we obtain
\[ Cn\theta_{\sqrt{n}/(2C)}(q)\ge\tilde\theta_n(q)\rightarrow\tilde\theta(q)>0,\]
since $\tau^{B_{C^2n}}_0\ge 4Cn\Rightarrow 0\not\in[A\cap B_{\sqrt{n}}]$. Indeed, since $\cU$ is not supercritical, we can find three or four stable directions containing the origin in their convex envelope, which guarantees that $[B_{\sqrt{n}}]\subset B_{\sqrt{Cn}}$ and inside this box sites will become infected at least one at a time. This proves that $\theta_{n}(q)\ge c/n^2$ for some $c>0$ and thus $q\le\qct$ by \eqref{eq:def:pct}. Hence, $\qct=\hat\qc$ and the proof of the first two items is complete.

Let us turn to the third one. As we already observed the occurrence of $E_n$ implies the existence of a site $x$ within distance $C/4$ of $\partial B_{n/2}$ with $\tau^{B_n}_x\ge n/C$. However, the event $\tau_x\ge n/C$ does not depend on sites outside $B_n$, so that it is the same as $\tau_x^{B_n}\ge n/C$ and the first one's probability is independent of $x\in B_{2n/3}$. Then the union bound gives
\[Cn\Pr_q(\tau_0\ge n/C)\ge \tilde\theta_{n}(q)\rightarrow\tilde\theta(q)>0.\]
Thus, for $q<\qct$ we have $\Pr_q(\tau_0>n)\ge c/n$ for some $c>0$ and in particular the first moment of $\tau_0$ is infinite, which completes the proof for subcritical models.

For $\cU$ critical or supercritical and $q>0$ it suffices to recall from Remark~\ref{rem:exp:decay} that $\Pr_q(\tau_0\ge n)$ decays exponentially, which immediately implies the exponential decay of $\tilde\theta_n(q)$ by the union bound as above and thus completes the proof (the second and third items being void for $\qct=0$).
\end{proof}

\subsection{Noise sensitivity---proof of Theorem~\ref{th:noise}}
We next use the algorithm we have to study noise sensitivity and prove Theorem~\ref{th:noise}. 

The harder part of the proof of Theorem~\ref{th:noise} relies on the following easy consequence of Theorem~1.8 of Schramm and Steif~\cite{Schramm10} and Theorem~1.9 of Benjamini, Kalai and Schramm~\cite{Benjamini99}.\footnote{The results of these papers are stated for $q=1/2$, but they are also valid for any fixed value of $0<q<1$. Moreover, the result does hold for the stronger Definition~\ref{def:noise}.}
\begin{thm}[\cites{Schramm10,Benjamini99}]
\label{th:SS}
Let $G_n$ be a sequence of cylinder events (depending on finitely many sites). If there exists a randomised algorithm determining the occurrence of $G_n$ with maximal revealment $\delta_n\to 0$, then the sequence is noise sensitive.
\end{thm}
The straightforward converses in Theorem~\ref{th:noise}, stated for completeness, follow from the next easy lemma.
\begin{lem}
\label{lem:Pete}
Let $G_n$ be a nested sequence of cylinder events such that $\bigcap_n G_n=G_\infty$ and $0<\Pr_q(G_\infty)<1$. Then $G_n$ are not noise sensitive.
\end{lem}
\begin{proof}
Firstly, $Var(\1_{G_n})\to Var(\1_{G_\infty})\in (0,1/4]$. Secondly, $\1_{G_n}\xrightarrow{L^2}\1_{G_\infty}$, so that for any $\delta>0$ there exists $n_\delta$ such that for all $n\ge n_\delta$ we have $\|\1_{G_n}-\1_{G_{n_\delta}}\|_{L^2}<\delta$. Finally, for any $\varepsilon>0$ the function $f\mapsto (x\mapsto \Ex[f(N_\varepsilon(x)) |x])$ is an $L^2$ contraction, so that for all $n\ge n_\delta$ we also have $\|\1_{N_\varepsilon(x)\in G_n}-\1_{N_\varepsilon(x)\in G_{n_\delta}}\|_{L^2}<\delta$. These three facts combined imply that it is sufficient to show that for any $\delta>0$ small enough and any $\varepsilon>0$ small enough depending on $\delta$ it holds that $Var(\1_{G_{n_\delta}})-Cov(\1_{x\in G_{n_\delta}},\1_{N_\varepsilon(x)\in G_{n_\delta}})<\delta$. But this is the case, as $G_{n_\delta}$ is a cylinder event, so that for $\varepsilon$ small enough $\Pr_q(\1_{x\in G_{n_\delta}}\neq \1_{N_\varepsilon(x)\in G_{n_\delta}})<\delta$. Hence,
\[\lim_{\varepsilon\to 0}\liminf_{n\to\infty}\frac{Cov(\1_{x\in G_n},\1_{N_\varepsilon(x)\in G_n})}{Var(\1_{G_n})}=1,\]
which concludes the proof by Definition~\ref{def:noise}.
\end{proof}
\begin{rem}
The consequences of Lemma~\ref{lem:Pete} can also be deduced easily from \cite{Benjamini99}*{Theorem~1.4}.
\end{rem}

\begin{proof}[Proof of Theorem~\ref{th:noise}]
Fix $0<q<1$. First assume that $\theta(q)>0$. Then by Lemma~\ref{lem:Pete} we have that the events $0\not\in[A\cap B_n]$ are not noise sensitive and then Theorem~\ref{th:SS} proves that no low-revealment algorithm exists. The proof in the case $\tilde\theta(q)>0$ that the events $E_n$ are not noise sensitive is analogous. Assume, on the contrary, that $\tilde\theta(q)=0$. Then Lemma~\ref{lem:exp:decay} provides an algorithm with revealment $\delta_n\to 0$, which completes the proof of the first two items of Theorem~\ref{th:noise}.

Finally, assume that $\theta(q)=\tilde\theta(q)=0$. Since $\theta(q)=0$ we also have $\Pr_q(\tau_0\ge n)\to 0$. Fix $\varepsilon>0$ and let $n$ be large enough so that we can find $n/C>k_0>C$ with $k_0<\varepsilon/(64C\Pr_q(\tau_0\ge n/C))$ and $\frac{2}{k_0}\sum_{m=0}^{2k_0}\tilde\theta_m(q)<\varepsilon$. Denote by $H_k$ the event that there exists $x$ at distance at most $C$ from $\partial B_{k}$ such that $\tau_x^{B_n}>n/C$.  Then by the union bound $\Pr_{q}(H_k)< 16Ck \Pr_q(\tau_0\ge n/C)<\varepsilon$ for $k<4k_0$.

We perform the same algorithm as in the proof of Lemma~\ref{lem:exp:decay}, but with $k$ chosen uniformly in $[3k_0,4k_0)$. When the first stage (exploration) of the algorithm stops we check if $H_k$ occurs, which is indeed known (witnessed by the set of inspected sites $S$). If it does, then we simply check all the remaining sites to determine if $0\in[A\cap B_n]$. The probability that this last step occurs is exactly $\Pr_{q}(H_k)<\varepsilon$. If $H_k$ does not occur, we know that $0\in[A\cap B_n]$ (since there are no finite stable healthy sets). We can then bound the revealment similarly to what we did in Lemma~\ref{lem:exp:decay}---we consider a site $y\in\partial B_l$ and take cases depending on its position. If $l\ge5k_0$, the revealment is at most $\varepsilon+\tilde\theta_{l-4k_0}(q)\le\varepsilon+\tilde\theta_{k_0}(q)<2\varepsilon$ and similarly for $l<2k_0$. For $2k_0\le l< 5k_0$ we average on $k$ as before to obtain a revealment bounded by $\varepsilon+\frac{2}{k_0}\sum_{m=0}^{2k_0}\tilde\theta_m(q)$. Hence, the maximal revealment is indeed bounded by $2\varepsilon$. Then, as previously, Theorem~\ref{th:SS} gives that $0\in[A\cap B_n]$ is noise sensitive, which concludes the proof.
\end{proof}

\subsection{Spectral gap and mean infection time of KCM}
To conclude our discussion of exponential decay, we turn to its applications to the KCM defined at the end of the introduction. Cancrini, Martinelli, Roberto and Toninelli~\cite{Cancrini08} proved the positivity of the spectral gap above $\qc$ for several specific models including OP, whose KCM counterpart is known as the North-East model. They also proved that the result holds for any model under an unhandy additional condition. We now use Theorem~\ref{th:exp:decay} together with their results to prove that for all KCM the gap is positive above $\qct$ and $0$ below and the mean infection time of the origin is finite and infinite respectively. It is very interesting to note that we will use the exponential decay of $\tilde\theta_n$ and not $\theta_n$, which does not suffice.

In order to link the spectral gap and the mean infection times we need the following simple facts from~\cite{Martinelli19} and~\cite{Cancrini09}.
\begin{lem}[Lemma~4.3~\cite{Martinelli19}, Theorem~4.7~\cite{Cancrini09}]
\label{lem:MT}
For all $0<q<1$ the mean infection time of the origin in the BP and the corresponding stationary KCM processes satisfy
\[\delta\Ex^{\mathrm{BP}}_{q}[\tau_0]\le\Ex^{\mathrm{KCM}}_{q}[\tau_0]\le\frac{T_{\mathrm{rel}}(q)}{q},\]
where $T_{\mathrm{rel}}$ is the inverse spectral gap of the KCM and $\delta>0$ is a sufficiently small constant.
\end{lem}

\begin{figure}
\begin{center}
\begin{tikzpicture}[line cap=round,line join=round,>=triangle 45,x=0.45cm,y=0.45cm]
%\clip(-14,-1) rectangle (16,9);
\fill[fill=black,pattern=vertical lines] (-1.75,0.95) -- (2.23,0.95) -- (-8.31,6.71) -- (-12.3,6.71) -- cycle;
\fill[fill=black,pattern=vertical lines] (-1.51,1.92) -- (2,0) -- (14,0) -- (10.49,1.92) -- cycle;
\fill[line width=0pt,fill=black,fill opacity=0.25] (14,0) -- (16,0) -- (14.25,0.95) -- (12.25,0.95) -- cycle;
\fill[line width=0pt,fill=black,fill opacity=0.25] (-14.05,7.66) -- (-12.3,6.71) -- (-10.3,6.71) -- (-12.05,7.66) -- cycle;
\draw (0,0)-- (16,0);
\draw (0,0)-- (-14.05,7.66);
\draw (-14.05,7.66)-- (1.95,7.66);
\draw (1.95,7.66)-- (16,0);
\draw (2.23,0.95)-- (-8.31,6.71);
\draw (-8.31,6.71)-- (-12.3,6.71);
\draw (-12.3,6.71)-- (-1.75,0.95);
\draw (-1.51,1.92)-- (2,0);
\draw (2,0)-- (14,0);
\draw (14,0)-- (10.49,1.92);
\draw (10.49,1.92)-- (-1.51,1.92);
\draw (-1.75,0.95)-- (2.23,0.95);
\draw (12.25,0.95)-- (14.25,0.95);
\draw (-10.3,6.71)-- (-12.05,7.66);
\draw (-14,7.5) node[anchor=north west] {$\mathbf{b}$};
\draw (15.2,1) node[anchor=north west] {$\mathbf{a}$};
\draw (2,0) node[anchor=north west] {$\varepsilon \mathbf{a}$};
\draw (-2.5,1) node[anchor=north west] {$\varepsilon \mathbf{b}$};
\end{tikzpicture}
\end{center}
\caption{Illustration of the definition of a renormalised site being good. The two hatched parallelograms become infected by the first condition, while the second one concerns the two shaded rhombi.}
\label{fig:gap}
\end{figure}
\begin{figure}
\begin{center}
\begin{tikzpicture}[line cap=round,line join=round,>=triangle 45,x=0.225cm,y=0.225cm]
%\clip(-30,-8) rectangle (30,8);
\fill[fill=black,pattern=vertical lines] (-1.75,0.95) -- (2.23,0.95) -- (-8.31,6.71) -- (-12.3,6.71) -- cycle;
\fill[fill=black,pattern=vertical lines] (-1.51,1.92) -- (2,0) -- (14,0) -- (10.49,1.92) -- cycle;
\fill[line width=0pt,fill=black,fill opacity=0.25] (0,0) -- (1.75,-0.95) -- (3.75,-0.95) -- (2,0) -- cycle;
\fill[line width=0pt,fill=black,fill opacity=0.25] (-2,0) -- (0,0) -- (-1.75,0.95) -- (-3.75,0.95) -- cycle;
\fill[fill=black,pattern=vertical lines] (12.3,-6.71) -- (16.28,-6.71) -- (5.73,-0.95) -- (1.75,-0.95) -- cycle;
\fill[fill=black,pattern=vertical lines] (-17.51,1.92) -- (-14,0) -- (-2,0) -- (-5.51,1.92) -- cycle;
\fill[fill=black,pattern=vertical lines] (-17.75,0.95) -- (-13.77,0.95) -- (-24.31,6.71) -- (-28.3,6.71) -- cycle;
\fill[fill=black,pattern=vertical lines] (12.53,-5.75) -- (16.05,-7.66) -- (28.05,-7.66) -- (24.53,-5.75) -- cycle;
\fill[fill=black,pattern=north west lines] (2,0) -- (0.25,0.95) -- (-1.75,0.95) -- (0,0) -- cycle;
\fill[fill=black,pattern=north east lines] (0,0) -- (2,0) -- (0.25,0.95) -- (-1.75,0.95) -- cycle;
\fill[fill=black,pattern=north west lines] (-14,0) -- (-15.75,0.95) -- (-17.75,0.95) -- (-16,0) -- cycle;
\fill[fill=black,pattern=north east lines] (-16,0) -- (-14,0) -- (-15.75,0.95) -- (-17.75,0.95) -- cycle;
\fill[fill=black,pattern=north east lines] (14.05,-7.66) -- (16.05,-7.66) -- (14.3,-6.71) -- (12.3,-6.71) -- cycle;
\fill[fill=black,pattern=north west lines] (16.05,-7.66) -- (14.3,-6.71) -- (12.3,-6.71) -- (14.05,-7.66) -- cycle;
\draw (0,0)-- (16,0);
\draw (0,0)-- (-14.05,7.66);
\draw (-14.05,7.66)-- (1.95,7.66);
\draw (1.95,7.66)-- (16,0);
\draw (2.23,0.95)-- (-8.31,6.71);
\draw (-8.31,6.71)-- (-12.3,6.71);
\draw (-12.3,6.71)-- (-1.75,0.95);
\draw (-1.51,1.92)-- (2,0);
\draw (2,0)-- (14,0);
\draw (14,0)-- (10.49,1.92);
\draw (10.49,1.92)-- (-1.51,1.92);
\draw (-1.75,0.95)-- (2.23,0.95);
\draw (12.3,-6.71)-- (16.28,-6.71);
\draw (16.28,-6.71)-- (5.73,-0.95);
\draw (5.73,-0.95)-- (1.75,-0.95);
\draw (1.75,-0.95)-- (12.3,-6.71);
\draw (-17.51,1.92)-- (-14,0);
\draw (-14,0)-- (-2,0);
\draw (-2,0)-- (-5.51,1.92);
\draw (-5.51,1.92)-- (-17.51,1.92);
\draw (-17.75,0.95)-- (-13.77,0.95);
\draw (-13.77,0.95)-- (-24.31,6.71);
\draw (-24.31,6.71)-- (-28.3,6.71);
\draw (-28.3,6.71)-- (-17.75,0.95);
\draw (12.53,-5.75)-- (16.05,-7.66);
\draw (16.05,-7.66)-- (28.05,-7.66);
\draw (28.05,-7.66)-- (24.53,-5.75);
\draw (24.53,-5.75)-- (12.53,-5.75);
\draw (2,0)-- (0.25,0.95);
\draw (0.25,0.95)-- (-1.75,0.95);
\draw (-1.75,0.95)-- (0,0);
\draw (0,0)-- (2,0);
\draw (0,0)-- (2,0);
\draw (2,0)-- (0.25,0.95);
\draw (0.25,0.95)-- (-1.75,0.95);
\draw (-1.75,0.95)-- (0,0);
\draw (-14,0)-- (-15.75,0.95);
\draw (-15.75,0.95)-- (-17.75,0.95);
\draw (-17.75,0.95)-- (-16,0);
\draw (-16,0)-- (-14,0);
\draw (-16,0)-- (-14,0);
\draw (-14,0)-- (-15.75,0.95);
\draw (-15.75,0.95)-- (-17.75,0.95);
\draw (-17.75,0.95)-- (-16,0);
\draw (14.05,-7.66)-- (16.05,-7.66);
\draw (16.05,-7.66)-- (14.3,-6.71);
\draw (14.3,-6.71)-- (12.3,-6.71);
\draw (12.3,-6.71)-- (14.05,-7.66);
\draw (16.05,-7.66)-- (14.3,-6.71);
\draw (14.3,-6.71)-- (12.3,-6.71);
\draw (12.3,-6.71)-- (14.05,-7.66);
\draw (14.05,-7.66)-- (16.05,-7.66);
\draw (-30.05,7.66)-- (-14.05,7.66);
\draw (16,0)-- (30.05,-7.66);
\draw (0,0)-- (14.05,-7.66);
\draw (-16,0)-- (-1.95,-7.66);
\draw (-16,0)-- (0,0);
\draw (-1.95,-7.66)-- (14.05,-7.66);
\draw (-16,0)-- (-30.05,7.66);
\draw (-3.75,0.95)-- (-1.75,0.95);
\draw (3.75,-0.95)-- (2,0);
\draw (14.05,-7.66)-- (30.05,-7.66);
\end{tikzpicture}
\end{center}
\caption{Infection procedure used to prove that if the top-right, bottom-right and top-left renormalised sites are good, the bottom-left one becomes entirely infected.}
\label{fig:gap2}
\end{figure}
\begin{proof}[Proof of Theorem~\ref{th:gap}]
Let $\cU$ be a (non-trivial) update family and without loss of generality assume that it contains a rule $U_0\subset\H_{-\pi/2+\delta}\cap\H_{-\pi/2-2\delta}$ for some $\delta>0$ sufficiently small such that $-\pi/2-\delta$ is a rational direction. Fix $q>\qct$ and $\varepsilon(\delta)>0$ and $\eta(\delta,\varepsilon)>0$ sufficiently small. The positivity of the gap is implied by Theorem~3.3 of~\cite{Cancrini08} if we can find a suitable renormalisation satisfying the following (see Definition 3.1~\cite{Cancrini08}).\footnote{The statement in~\cite{Cancrini08} is given for square boxes, but generalises without change.}
\begin{itemize}
\item[(a)] Each renormalised site is good with probability at least $1-\varepsilon$.
\item[(b)] If the renormalised sites $(0,1)$, $(1,0)$ and $(1,1)$ are all good, then
\[[A\cap (\{\mathbf{a},\mathbf{b},\mathbf{a}+\mathbf{b}\}+B')]\supset B',\] where $\mathbf{a}$ and $\mathbf{b}$ are the two base vectors of the renormalisation and $B'$ is the renormalisation box---the parallelogram generated by $\mathbf{a}$ and $\mathbf{b}$ i.e.
\[B'=([0,1)\cdot\mathbf{a})+([0,1)\cdot\mathbf{b}),
\]
where we use the notation $C+D=\{c+d,\;c\in C,d\in D\}$.
\end{itemize}

Set $\mathbf{a}=(n,0)$ and $\mathbf{b}=n(\cos(-\pi+\delta),\sin(-\pi+\delta))$ for $n(\eta)$ sufficiently large. We call the renormalised site $0$ \emph{good} if the following all hold (see Figure~\ref{fig:gap}) and we extend the definition to any site by translation.
\begin{itemize}
\item For all $x$ in the parallelograms $[\varepsilon,1-\varepsilon]\cdot\mathbf{a}+ [0,2\varepsilon]\cdot\mathbf{b}$ and $[\varepsilon ,1-\varepsilon]\cdot \mathbf{b}+ [0,2\varepsilon]\cdot\mathbf{a}$ it holds that $\tau_x^{B'}< \eta n$.
\item For all $x$ in the rhombus $[1-\varepsilon,1)\cdot\mathbf{a}+[0,\varepsilon]\cdot\mathbf{b}$ it holds that $\tau_x^{B'}<\eta n$ if we impose infected boundary condition on $[1,1+2\varepsilon]\cdot\mathbf{a}+[0,1-\varepsilon]\cdot\mathbf{b}$ and healthy on the rest of $\Z^2\setminus B'$. Also the symmetric condition holds for the rhombus $[1-\varepsilon,1)\cdot\mathbf{b}+[0,\varepsilon]\cdot\mathbf{a}$.
\end{itemize}
Condition (b) on the renormalisation is easily checked from this definition, using only the rule $U_0$ (see Figure~\ref{fig:gap2}). Indeed, all hatched regions become infected by the first condition, so that the double hatched rhombi are infected by $U_0$. Finally, the shaded rhombi become infected by the second condition, since the infected boundary condition is already met. The renormalised site considered is then entirely infected using $U_0$. Thus, we only need to check that a renormalised site is good with probability at least $1-\varepsilon$.

Since the conditions concern $O(n^2)$ sites, by symmetry and monotonicity it suffices to observe that
\[\Pr_q\left(\tau_0^{[-C\eta n,C\eta n]\times [0,C\eta n]}\ge\eta n\right)\]
decays exponentially with $n$. Indeed, for this event to occur, there must exist a path of sites $x_0,\ldots, x_{\lceil n\eta\rceil}=0$ with $x_{i}-x_{i+1}\in U_0$ and $\tau_{x_i}^{[-C\eta n,C\eta n]\times [0,C\eta n]}\ge i$ for all $0\le i<\eta n$, which in particular means that $E_{\eta^2 n}$ translated by $x_0$ occurs. Hence, using the first item of Theorem~\ref{th:exp:decay} and the union bound we obtain the desired result and thereby the spectral gap is strictly positive. By Lemma~\ref{lem:MT} this implies that the mean infection time of the KCM is finite.

Finally, by Theorem~\ref{th:exp:decay} for $q<\qct$ the mean infection time of BP is infinite, so Lemma~\ref{lem:MT} shows that in this regime the spectral gap is $0$ and the mean infection time of the KCM is infinite.
\end{proof}

\section{Open problems}
\label{sec:open}
To conclude, let us mention some interesting open problems related to this work besides its direct extensions based on GOP.
\subsection{Simplifications}
We next mention the two prime conjectures which would greatly simplify the statements of our results besides being interesting on their own. We start with the uniqueness of the transition.
\begin{conj}
\label{conj:qc:qct}
For all update families we have
\[\qc=\qct.\]
\end{conj}
We should note that, the Kahn--Kalai--Linial theorem~\cite{Kahn88} tells us that (up to replacing the box by the torus as in~\cite{Balogh03} or adapting the technique of~\cite{Duminil-Copin18}) $\theta_n(q)$ decays at least like $n^{-\varepsilon(q-\qc)}$ above criticality and Theorem~\ref{th:exp:decay} establishes that below $\qct$ it decays at most like $n^{-2}$. As it is commonly the case, it is likely that breaching this gap will prove difficult.

As mentioned earlier if one proves the slightly stronger property
\begin{equation}
\label{eq:conj:stronger}
\tilde\theta(q)>0\Rightarrow\theta(q)>0,
\end{equation}
which implies Conjecture~\ref{conj:qc:qct}, then Theorem~\ref{th:noise} exhausts the noise sensitivity problem for subcritical BP at least for the most natural event $0\in[A\cap B_n]$, which we consider since there is no obvious choice of ``crossing'' event. Indeed, in view of Question \ref{ques:torus} below, it is not clear whether it is relevant to consider the event of complete infection on the torus. Also in the light of Theorem~\ref{th:noise} the converse implication of \eqref{eq:conj:stronger} is not uninteresting at $\qct$.

Secondly, it would be practical to know if the complication of taking limits in Definition~\ref{def:diff} is necessary. We suspect that this is never the case.
\begin{ques}
What are the continuity properties of the function $(u,\theta)\mapsto d_u^\theta$?
\end{ques}

\subsection{Torus}
Although the most natural setting for subcritical models is the infinite volume quantity $\theta$, which is approximated by its restriction to boxes $\theta_n$, another common choice in order to avoid boundary issues is to consider the torus $\bbT_n=(\Z/n\Z)^2$. Indeed, results for critical and supercritical models are meaningful in this setting and are essentially equivalent to the law of the infection time in infinite volume~\cite{Bollobas15}. Yet, for subcritical models the mechanism of infection is rather different---instead of rare large droplets that grow easily we have common droplets which only manage to grow with a lot of help. Owing to this it is not clear how quantities on the torus relate to those on the entire grid. We should mention that most of our results carry through if all is defined on the torus, but it is interesting to note that not even the next question seems to have been answered yet.
\begin{ques}
\label{ques:torus}
Does one have that for all subcritical families
\[\qc=\liminf_{n}\{q,\,\Pr_q([A]_{\bbT_n}=\bbT_n)\ge 1/2\},\]
where the closure is taken with respect to the BP process on the torus and $A$ is a random subset of $\bbT_n$ of density $q$?
\end{ques}

\section*{Acknowledgements}
\addcontentsline{toc}{section}{\numberline{}Acknowledgements}
This work was supported by ERC Starting Grant 680275 MALIG. The author would particularly like to thank Cristina Toninelli for countless discussions, useful suggestions and guidance throughout the preparation of this work. Thanks are also due to G\'abor Pete for useful comments on noise sensitivity and on the original proof of Theorem \ref{th:noise}; to Justin Salez for careful proofreading and to Vincent Tassion for enlightening discussions. Finally, we thank the anonymous referees for extensive, meticulous and useful comments, which helped greatly improve the presentation.

\addcontentsline{toc}{section}{\numberline{}References}
\bibliographystyle{plain}
\bibliography{Bib}
\end{document}